\documentclass[11pt]{article}
\usepackage{amscd}
\usepackage{amsmath}
\usepackage{latexsym}
\usepackage{amsfonts}
\usepackage{amssymb}
\usepackage{amsthm}
\usepackage{graphicx}
\usepackage{verbatim}
\usepackage{mathrsfs}
\usepackage{enumerate}
\usepackage{hyperref}
\usepackage{fullpage}
\usepackage{color}

\theoremstyle{plain}
\newcommand{\norm}[1]{\left\|#1\right\|}
\theoremstyle{definition}\newtheorem{definition}{Definition}[section]
\theoremstyle{plain}\newtheorem{theorem}{Theorem}[section]
\theoremstyle{plain}\newtheorem{lemma}[theorem]{Lemma}
\theoremstyle{plain}
\theoremstyle{plain}\newtheorem{proposition}[theorem]{Proposition}
\theoremstyle{remark}\newtheorem{remark}{Remark}[section]

\newcommand{\NN}{\mathbb{N}}

\newcommand{\RR}{\mathbb{R}}

\DeclareMathOperator{\intd}{d}
\DeclareMathOperator{\dtd}{\frac{d}{d{\emph t}}}
\allowdisplaybreaks
\begin{document}
\title{Global well-posedness for axisymmetric Boussinesq system with horizontal viscosity}
\author{Changxing Miao$^1$ and  Xiaoxin Zheng$^2$\\
\\
        \small{$^{1}$ Institute of Applied Physics and Computational Mathematics,}\\
        \small{P.O. Box 8009, Beijing 100088, P.R. China.}\\
        \small{(miao\_{}changxing@iapcm.ac.cn)}\\
        \small{$^2$  The Graduate School of China Academy of Engineering Physics,}\\
        {\small P.O. Box 2101, Beijing 100088, P.R. China.}\\
        {\small{(xiaoxinyeah@163.com)}}}

\date{}
\maketitle

\begin{abstract}
In this paper, we are concerned with the tridimensional anisotropic Boussinesq equations which can be described by
\begin{equation*}
  \left\{\begin{array}{ll}
       \partial_{t}u+u\cdot\nabla u-\kappa\Delta_{h} u+\nabla \Pi=\rho e_{3},\quad(t,x)\in\mathbb{R}^{+}\times\mathbb{R}^{3},\\
       \partial_{t}\rho+u\cdot\nabla \rho=0,\\
       \text{div}\,u=0.
    \end{array}\right.
\end{equation*}
Under the assumption that the support of the axisymmetric initial data $\rho_{0}(r,z)$
 does not intersect the axis $(Oz)$, we prove the global well-posedness for this system  with axisymmetric initial data. We first show the growth of the quantity $\frac\rho r$ for large time by taking advantage of characteristic of transport equation. This growing property together with the horizontal smoothing effect enables us to establish  $H^1$-estimate of the velocity via the $L^2$-energy estimate of velocity and the Maximum principle of density. Based on this, we further   establish the estimate for the quantity
$\|\omega(t)\|_{\sqrt{\mathbb{L}}}:=\sup_{2\leq p<\infty}\frac{\norm{\omega(t)}_{L^p(\mathbb{R}^3)}}{\sqrt{p}}<\infty$
which implies $\|\nabla u(t)\|_{\mathbb{L}^{\frac{3}{2}}}:=\sup_{2\leq p<\infty}\frac{\norm{\nabla u(t)}_{L^p(\mathbb{R}^3)}}{p\sqrt{p}}<\infty$. However, this regularity for the flow admits forbidden singularity since $ \mathbb{L}$ (see \eqref{eq-kl} for the definition) seems be the minimum space for the gradient vector field $u(x,t)$ ensuring uniqueness of flow. To bridge this gap, we exploit the space-time estimate about $ \sup_{2\leq p<\infty}\int_0^t\frac{\|\nabla u(\tau)\|_{L^p(\mathbb{R}^3)}}{\sqrt{p}}\mathrm{d}\tau<\infty$ by making good use of the horizontal smoothing effect and micro-local techniques. The global well-posedness for the large initial data is achieved by establishing a new type space-time logarithmic inequality.

\end{abstract}
\section{Introduction}
\setcounter{section}{1}\setcounter{equation}{0}
The Boussinesq system is used as a toy model in the dynamics of the ocean or of the atmosphere, and play an important role in the study of
Raleigh-Bernard convection.  One may refer to \cite{J-P} for more details about its physical background. It takes the form:
\begin{equation}\label{full}
   \left\{\begin{array}{ll}
       \partial_{t}u+u\cdot\nabla u-\kappa\Delta u+\nabla \Pi=\rho e_{n},\quad(t,x)\in\mathbb{R}^{+}\times\mathbb{R}^{n},\quad n=2,3,\\
       \partial_{t}\rho+u\cdot\nabla \rho-\nu\Delta\rho=0,\\
       \text{div}\,u=0,\\
       (u,\rho)|_{t=0}=(u_{0},\rho_{0}),
    \end{array}\right.
\end{equation}
where, the velocity  $u=(u_1,\cdots,u_n)$ is a vector field  and the scalar unknown $\rho$ denotes quantity such as the concentration of a chemical substance or the temperature variation in
a gravity fields, in which case $\rho e_n$ represents the buoyancy force. And the nonnegative parameters $\kappa$ an $\nu$ stands for the viscosity and the molecular diffusion respectively. In addition, the scalar function $\Pi$ is pressure which can be recovered from the unknowns $u$ and $\rho$ via Riesz operator.

This system have been intensively studied due to their physical background and mathematical significance. In dimension two, the standard energy method enables us to establish  the global existence of regular solutions for the case where $\nu$ and $\kappa$ are nonnegative constants. But, for the inviscid Boussinesq \mbox{system \eqref{full}}, the global well-posedness for some nonconstant $\rho_0$ is still an challenge open problem.  When $\nu$ is a positive constant and  $\kappa=0$; or $\nu=0$ and $\kappa$ is a positive constant,  the global well-posedness was  independently  obtained in \cite{ha,hou-li} for the two-dimensional Boussinesq system, see also \cite{hk} for the global well-posedness in the critical spaces. For the fractional case, Hmidi, Keraani and Rousset \cite{hkr} showed the  global well-posedness  for the critical case by using a hidden cancellation given by the coupling. Moreover, for the case where fractional viscosity and thermal diffusion the fractional powers satisfy mild condition, the global results on the two-dimensional Boussinesq equations  were obtained in \cite{JMWZ,CMX}.

For the tri-dimensional Boussinesq equations, R. Danchin and M. Paicu \cite{DP2008} showed the global existence of weak solution for $L^2$-data and the global well-posedness for small initial data. They  \cite{DP2008-2} also obtained a existence and uniqueness
result for small initial data belonging to some critical Lorentz spaces. But there is little study about the global well-posednss result for large initial data, even for the tri-dimensional Navier-Stokes equations. Inspired by the study of Navier-Stokes equations for large data in special case, more recent works target to consider the tri-dimensional axisymmetric Boussinesq system without swirl case. In \cite{A-H-K0}, a global existence and uniqueness result for the following system
\begin{equation}
      \label{full-bs}
\left\{ \begin{array}{ll}
\partial_{t}u+u\cdot\nabla u-\kappa\Delta u+\nabla \Pi=\rho e_3,\quad (t,x)\in \RR^+\times\RR^3,\\
\partial_{t}\rho+u\cdot\nabla\rho=0,\\
\textnormal{div}\,  u=0,\\
  (u,\rho)|_{t=0}=(u_{0},\rho_{0}),
\end{array} \right.
\end{equation}
 was obtained by establishing the following quadratic growth estimate
    \begin{equation}\label{growth-2}
   \Big \|\frac{\rho}{r}(t)\Big\|_{L^2}\leq  \Big \|\frac{\rho_0}{r}\Big\|_{L^2}+C_0\Big\|\frac{u_r}{r}\Big\|_{L^1_tL^\infty}\big(1+\|u\|_{L^1_tL^\infty}\big),
\end{equation}
under assumption that the support of the initial density does not intersect the  axis $r=0$.
From that time on, much effects has been made to show the global well-posedness for the tri-dimensional axisymmetric Boussinesq system without swirl case, when
the dissipation only occurs one equation or is present only in one direction (anisotropic dissipation). In a series of paper \cite{hrou1,hrou},  T. Hmidi and F. Rousset \cite{hrou1} proved the global well-posedness for the Navier-Stokes-Boussinesq system by virtue of the structure of the coupling between two equations of \eqref{full} with $\nu=0$. In \cite{hrou}, they also showed the global well-posedness for the tridimensional Euler-Boussinesq system with axisymmetric initial data without swirl. And their proofs strongly relies on the fact the dissipation occurs in three directions.

As pointed out by J.-Y. Chemin et al in \cite{Cdgg},  the anisotropic dissipation assumption is natural and physical. In fact, in certain regimes and after suitable rescaling, the vertical dissipation (or the horizontal dissipation) is negligible  as compared to the horizontal dissipation (or the vertical dissipation). In the past years, there are several works devoted to study of the two-dimensional Boussinesq system with anisotropic dissipation. When the horizontal viscosity occurs in only one equation,  the global well-posedness result for the two-dimensional Boussinesq system  was obtained in \cite{dp2}. Moreover,  A. Adhikari, C. Cao and J. Wu also established some global results under various assumption for the two-dimensional Boussinesq system with dissipation occurs in vertical direction in a series of recent papers \cite{acw,acw1}. More recently, \mbox{C. Cao and J. Wu } \cite{cw} successfully
proved the global well-posedness for the two-dimensional  Boussinesq system \eqref{full} in terms of a Log-type inequality
\begin{equation*}
\|u_2\|_{L^\infty}\leq C\|u_2\|_{\sqrt{L\log L}}\Big(\log\big(e+\|u_2\|_{H^2}\big)\log\log\big(e+\|u_2\|_{H^2}\big)\Big)^{\frac{1}{2}}.
\end{equation*}together with  a  control of $\norm{ u_2}_{\sqrt{L\log L}}$, where the space $
\sqrt{L\log L}$ stands for the space of functions $f$ in $\cap_{2\leq p<\infty}L^p$ such that
\begin{equation*}
\|f\|_{\sqrt{L\log L}}:=\sup_{2\leq p<\infty}\big({p\log p}\big)^{-\frac{1}{2}}\|f\|_{L^{p}}< \infty.
 \end{equation*}
 In addition, under the assumption that the initial data is  axisymmetric without swirl, we stated the global well-posedness in \cite {MZ2012} for the tridimensional Boussinesq equations with horizontal viscosity and diffusion by using a losing estimate with vector fields lying in Log-Lipschitz and establishing the algebraic identity
\begin{equation}\label{algebraic-idenity}
\frac{u_{r}}{r}=\partial_{z}\Delta^{-1}\left(\frac{\omega_\theta}{r}\right)-2\frac{\partial_r}{r}\Delta^{-1}\partial_z\Delta^{-1}\left(\frac{\omega_\theta}{r}\right)\cdot
\end{equation}

In the presented paper, we take effect to investigate the global well-posedness  for tridimensional Boussinesq system with horizontal viscosity
in the whole space  with axisymmetric initial data. This system is described as follows:
\begin{equation}\label{bs}
   \left\{\begin{array}{ll}
      \partial_{t}u+u\cdot\nabla u-\kappa\Delta_{h} u+\nabla \Pi=\rho e_{3},\quad(t,x)\in\mathbb{R}^{+}\times\mathbb{R}^{3},\\
       \partial_{t}\rho+u\cdot\nabla \rho=0,\\
       \text{div}u=0,\\
       (u,\rho)|_{t=0}=(u_{0},\rho_{0}),
       \tag{HBS}
    \end{array}\right.
\end{equation}
where $\Delta_{h}:=\partial^{2}_{1}+\partial^{2}_{2}$. In the following parts, we assume that $\kappa=1$ for the sake of convenience.

 First of all, let us recall some algebraic and geometric properties of the axisymmetric vector fields (cf. \cite{hrou1,MZ2012}) and
  discuss the special structure of the vorticity of   \eqref{bs}.
Let  $u$ is an axisymmetric vector field without swirl, that is, $u(t,x)=u_{r}(r,z)e_{r}+u_{z}(r,z)e_{z}$.
Then a simple calculation  yields that the vorticity $\omega:=\text{curl}\, u$ of the vector
field has the form
\begin{equation*}
\omega=(\partial_{z}u_{r}-\partial_{r}u_{z})e_\theta:=\omega_\theta e_\theta,
\end{equation*}
and
\begin{equation*}
u\cdot\nabla=u_{r}\partial_{r}+u_{z}\partial_{z},\quad \text{div}\,u=\partial_{r}u_{r}+\frac{u_{r}}{r}+\partial_{z}u_{z}\quad \text{and}\quad\omega\cdot \nabla u=\frac{u_{r}}{r}\omega
\end{equation*}
in the cylindrical coordinates. As a consequence, the vorticity $\omega$ solves
\begin{equation} \label{tourbillon-0}
\partial_t \omega +u\cdot\nabla\omega-\Delta_{h}\omega
 =-\partial_{r}\rho e_{\theta}+\frac{u_r}{r}\omega.
\end{equation}
This together with the fact that $\Delta_{h}=\partial_{rr}+\frac{1}{r}\partial_{r}$ in the cylindrical coordinates enables us to conclude  the quantity $\omega_
\theta$ satisfies
\begin{equation}\label{tourbillon}
\partial_t \omega_\theta +u\cdot\nabla\omega_\theta-\Delta_{h}\omega_\theta
+\frac{\omega_\theta}{r^2} =-\partial_{r}\rho+\frac{u_r}{r}\omega_\theta.
\end{equation}

The target of this paper is to study the global existence and the uniqueness for the system \eqref{bs} with
  axisymmetric  initial data, which means that the velocity $u_0$ is assumed to be an axisymmetric vector
  field without swirl  and the density $\rho_0$ depends only on $(r,z)$.  Now we shall briefly discuss the difficulties  and outline the main ingredient in our proof. First, the  quadratic growth estimate \eqref{growth-2} which
plays the key role in the proof of \cite{A-H-K0} does not work for the system \eqref{bs} due to the absence of vertical viscosity. Rough speaking, the difficulty
arises in dealing with vorticity equation.
Taking the $L^{2}$-inner product of \eqref{tourbillon}  with $\omega_\theta$ and integrating by parts, we readily get
  \begin{equation*}
  \frac12\dtd \norm{\omega_{\theta}(t)}^{2}_{L^{2}}+\norm{\nabla_h\omega_{\theta}(t)}^{2}_{L^{2}}+\Big\|\frac{\omega_{\theta}}{r}(t)\Big\|^{2}_{L^{2}}
  =\int_{\mathbb{R}^3}\frac{u_{r}}{r}\omega_\theta\omega_\theta \mathrm{d}x+\int_{\mathbb{R}^3} \rho\frac{\omega_\theta}{r} \mathrm{d}x+\int_{\mathbb{R}^3} \rho\partial_r\omega_\theta \mathrm{d}x.
  \end{equation*}
Taking advantage of the anisotropic inequality of Lemma \ref{lema.2}, the first integral term in the right side of the above equality can be bounded by
  \begin{equation*}
  C\|u\|_{L^2}\|\omega_{\theta}\|^{2}_{L^2}+\norm{\omega_\theta\over r}^{2}_{L^2}\norm{\nabla_h\Big(\frac{\omega_\theta}{ r}\Big)}^{2}_{L^2}
+\frac{1}{4}\|\nabla_h\omega_\theta\|^{2}_{L^2}.
  \end{equation*}
  On the other hand, the unknown $\frac{\omega_\theta}{r}$ satisfies the following equation
\begin{equation*}
\big(\partial_t+u\cdot\nabla\big)\frac{\omega_\theta}{r}-\big(\Delta_{h}+{{2 \over r}}\partial_r\big) \frac{\omega_\theta}{r} =-\frac{\partial_r\rho}{r}.
\end{equation*}
In a similar fashion as in \cite{A-H-K0}, one can conclude by the virtue of the
   estimate \eqref{growth-2} that
\begin{equation*}
\big\|{\omega_\theta\over r}(t)\big\|_{L^2}^2
+\frac12\int_0^t\big\|\nabla_{h}\Big({\omega_\theta\over r}\Big)(\tau)\big\|_{L^2}^2\intd\tau
\leq C_0(1+t^{5})\Big(1+\int_0^{t}\big\|\frac{\omega_\theta}{r}(\tau)\big\|_{L^2}^2\mathrm{d}\tau\Big)+\frac14\int_0^{t}\|\nabla_{h} \omega(\tau)\|^2_{L^2}\mathrm{d}\tau.
\end{equation*}
From this, it seems impossible to use the quantities in the left side of the above inequality to control the integral term $\int_0^t\|{\omega_\theta\over r}(\tau)\|^{2}_{L^2}\|\nabla_h(\frac{\omega_\theta}{ r})(\tau)\|^{2}_{L^2}\intd\tau$.
 This require us to refine this quadratic growth estimate to make up for the shortage of vertical diffusion.
To do this,  we establish the following estimate
\begin{equation*}
    \Big\|\frac{\rho}{r}(t)\Big\|_{L^2}\leq C\Big\|\frac{\rho_0}{r}\Big\|_{L^2}+ C\Big\|\frac{u_r}{r}\Big\|_{L^1_tL^\infty}\Big(1+\int_0^t\norm{\nabla_h u(\tau)}^{\frac{1}{2}}_{L^2}\norm{\nabla_h
    \omega(\tau)}^{\frac{1}{2}}_{L^{2}}\mathrm{d}\tau\Big).
\end{equation*}
 by deeply using the axisymmetric structure and the incompressible condition.   As an consequence, we have
   \begin{align*}
\big\|{\omega_\theta\over r}(t)\big\|_{L^2}^2
+\frac12\int_0^t\big\|\nabla_{h}\Big({\omega_\theta\over r}\Big)(\tau)\big\|_{L^2}^2\intd\tau
\leq&
C_0(1+t^{5})\Big(1+\int_0^{t}\big\|\frac{\omega_\theta}{r}(\tau)\big\|_{L^2}^2\mathrm{d}\tau\Big)\nonumber\\
&+\frac14\Big(\int_0^{t}\|\nabla_{h} \omega(\tau)\|^2_{L^2}\mathrm{d}\tau\Big)^{\frac12}.
\end{align*}
Consequently, we obtain the estimate of
\begin{equation*}
\|\omega_\theta(t)\|^{2}_{L^2}+\int_0^t\|\nabla_h\omega_\theta(\tau)\|^{2}_{L^2}\intd\tau+\big\|{\omega_\theta\over r}(t)\big\|_{L^2}^4
+\Big(\int_0^t\big\|\nabla_{h}\Big({\omega_\theta\over r}\Big)(\tau)\big\|_{L^2}^2\intd\tau\Big)^2.
\end{equation*}
   This entails us to obtain the estimate of $\|\omega(t)\|_{\sqrt{\mathbb{L}}}:=\sup_{2\leq p<\infty}\frac{\norm{\omega(t)}_{L^p(\mathbb{R}^3)}}{\sqrt{p}}$, which together with the well-known fact $$\|\nabla u\|_{L^p}\leq C\frac{p^2}{p-1}\|\omega\|_{L^p}\quad\text{with}\quad p\in]1,\infty[$$ gives $\|\nabla u(t)\|_{\mathbb{L}^{\frac{3}{2}}}:=\sup_{2\leq p<\infty}\frac{\norm{\nabla u(t)}_{L^p(\mathbb{R}^3)}}{p\sqrt{p}}<\infty$.
    Unfortunately, the function $p\sqrt{p}$ does not belong to the dual Osgood modulus of continuity, which prevents us  trying to obtain higher-order estimates of $(\rho,u)$, where an dual Osgood modulus of continuity $\omega(p)$ is the non-decreasing function satisfying $\int_a^\infty\frac{1}{\omega(\tau)}\intd\tau=\infty$ for some $a>0$. To bridge the gap between the dual Osgood modulus of continuity and $\mathbb{L}^{\frac{3}{2}}$, inspired by the Boot-Strap argument, we exploit the space-time estimate about $\sup_{2\leq p<\infty}\int_0^t\frac{\|\nabla u(\tau)\|_{L^p(\mathbb{R}^3)}}{\sqrt{p}}\mathrm{d}\tau$ based on the estimate of $\|\nabla u(t)\|_{\mathbb{L}^{\frac32}}$ by making good use of the horizontal smoothing effect and micro-local techniques. Combining this with a new type space-time logarithmic inequality established in Section \ref{Sec-2} entails us to obtain the desired result.

    Before stating our
  main result we denote by $\Pi_z$ the orthogonal projector over the axis $(Oz).$  We define the distance from a point $x$ to a subset $A\subset\RR^3$ by
$$
\mathrm{d}(x,A):=\inf_{y\in A}\|x-y\|,
$$
where $\|\cdot\|$ is the usual Euclidian norm.  The distance between two subsets $A$ and
$B$ of $\RR^3$ is defined by
$$
\mathrm{d}(A,B):=\inf_{x\in A,\, y\in B}\|x-y\|.
$$
And $
\text{diam }A=\sup_{x,\,y\in A}\|x-y\|
$ denotes the diameter of a bounded subset $A\subset \RR^3$.
Moreover, let us introduce  ${\mathbb{L}}^{a}(\mathbb{R}^3)(a\in[0,1])$ of those function $f$ which belong to every space $L^{p}(\mathbb{R}^3)$ with $2\leq p<\infty$ and satisfy
\begin{equation}\label{eq-kl}
\norm{f}_{\mathbb{L}^{a}(\mathbb{R}^3)}:=\sup_{p\geq 2}\frac{\norm{f}_{L^{p}(\mathbb{R}^3)}}{{p}^{a}}<\infty.
\end{equation}
We denote ${\mathbb{L}^{\frac12}(\mathbb{R}^3)}$ by $\sqrt{{\mathbb{L}}}(\mathbb{R}^3)$ for the sake of simplicity.

Our result reads as follows.
  \begin{theorem}\label{global}
   Assume that $u_0\in H^{1}(\mathbb{R}^3)$ be an axisymmetric vector field with zero divergence, and its vorticity satisfies   $\frac{{\omega_0}}{r}\in L^2(\mathbb{R}^3)$ and $\partial_z\omega_0\in L^2$.
 Let $\rho_0\in H^1(\mathbb{R}^3)\cap L^\infty(\RR^3)$ depending only on $(r,z)$ and such that $\hbox{\rm Supp }\rho_0$
 does not intersect the axis $(Oz)$ and $\Pi_z(\hbox{\rm Supp }\rho_0)$ is a compact set. Then the Boussinesq system \eqref{bs} has a unique global solution $(\rho,u)$ such that
\begin{align*}
u\in{C}\big(\RR^{+};H^{1}(\mathbb{R}^3)\big),\quad \nabla_h u\in L^2_{\rm loc}\big(\RR^+;H^1(\RR^3)\big),\quad\partial_z\omega,\,\frac{\omega}{r}\in C\big(\RR^+;L^2(\mathbb{R}^3)\big), \\\nabla_h\partial_z\omega,\,\nabla_h\Big(\frac{\omega}{r}\Big)\in L^2_{\textnormal{loc}}\big(\RR^+;L^2(\mathbb{R}^3)\big),
\quad\rho\in L^\infty\big(\RR^+;L^\infty(\RR^3)\big), \quad\rho\in{C}\big(\RR^{+};H^{1}(\mathbb{R}^3)\big),
\end{align*}
$$\quad \nabla u\in L^1_{\rm loc}\big(\RR^+;L^\infty(\RR^3)\big).
$$
\end{theorem}
\begin{remark}
Our proof strongly relies on the growth  estimate of $\frac\rho r$  and the horizontal smoothing effect.
\end{remark}

The rest of the paper is organized as follows. In Section \ref{Sec-2}, we  review  Littlwood-Paley theory and  establish a new type space-time logarithmic inequality which is an important ingredient in the proof of \mbox{Theorem \ref{global}}.  Next, we study analytic properties
of the flow associated to an axisymmetric vector field. In Section \ref{Sec-3}, we  obtain a priori estimates for sufficiently smooth solutions of the system \eqref{bs} by using the procedure that we have just described in introduction. \mbox{Section \ref{Sec-4}} is devoted to the proof of \mbox{Theorem \ref{global}}. Finally, an appendix is devoted to two useful lemmas.
  \section{Preliminaries}\label{Sec-2}
  \setcounter{section}{2}\setcounter{equation}{0}
 In the first subsection, we first provide the definition of some function spaces and properties based on the so-called Littlewood-Paley decomposition  that will be used constantly in the following sections. Next, we give a space-time logarithmic inequality in view of the low-high decomposition techniques.  In the last subsection, we main  establish the grow estimate of the quantity $\frac\rho r$ by taking advantage of some geometric and analytic properties of the  generalized flow map associated to  an axisymmetric vector field.
\subsection {Littlewood-Paley Theory and a space-time logarithmic inequality}

Let $(\chi,\varphi)$ be a couple of smooth functions with  values in $[0,1]$
such that $\text{Supp}\,\chi\subset\big\{\xi\in\mathbb{R}^{n}\big|\,|\xi|\leq\frac{4}{3}\big\}$,
$\text{Supp}\,\varphi\subset\big\{\xi\in\mathbb{R}^{n}\big|\,\frac{3}{4}\leq|\xi|\leq\frac{8}{3}\big\}$ and
\begin{equation*}
    \chi(\xi)+\sum_{j\geq0}\varphi(2^{-j}\xi)=1\quad \text {for  each}\quad\xi\in \mathbb{R}^{n}.
\end{equation*}
For every $u\in \mathcal{S}'(\mathbb{R}^{n})$, we define the littlewood-Paley operators as follows:
\begin{equation*}
 {S}_{j}u:=\chi(2^{-j}D)u\quad\text{and}\quad   {\Delta}_{j}u:=\varphi(2^{-j}D)u\quad \text{for all}\quad j\geq0.
\end{equation*}
From this, it is easy to verify  that
\begin{equation*}
    u=\sum_{j\geq-1}{\Delta}_{j}u,\quad\text{in}\quad\mathcal{S}'(\mathbb{R}^{n}),
\end{equation*}
and
\begin{equation*}
    {\Delta}_{j}{\Delta}_{j'}u\equiv 0\quad \text{if}\quad |j-j'|\geq 2.
\end{equation*}
Next, we recall the classical  Bernstein lemma which will be useful throughout this paper (\mbox{ cf. \cite{che1}}).
\begin{lemma}[Bernstein]\label{lb}\;
Let $1\leq p\leq q<\infty$ and $u\in L^p(\RR^n)$. There exists a positive constant $C$ such that for $j,\,k\in\NN$, we have
\begin{equation*}
\sup_{|\alpha|=k}\|\partial ^{\alpha}S_{j}u\|_{L^q(\RR^n)}
\leq
C^k\,2^{j\bigl(k+n(\frac{1}{p}-\frac{1}{q})\bigr)}\|S_{j}u\|_{L^p(\RR^n)},
\end{equation*}
and
\begin{equation*}
\ C^{-k}2^
{qk}\|{\Delta}_{j}u\|_{L^p(\RR^n)}
\leq
\sup_{|\alpha|=k}\|\partial ^{\alpha}{\Delta}_{j}u\|_{L^p(\RR^n)}
\leq
C^k2^{jk}\|{\Delta}_{q} u\|_{L^p(\RR^n)}.
\end{equation*}
\end{lemma}
Let us now introduce Bony's
decomposition  (see for example \cite{bcd}) which is a basic tool of  the para-differential calculus. Specifically, one can split
a product $uv$ into three parts as follows:
\begin{equation*}
uv=T_u v+T_v u+\mathcal{R}(u,v),
\end{equation*}
where
\begin{equation*}
T_u v=\sum_{q}S_{q-1}u\Delta_q v, \quad\text{and}\quad \mathcal{R}(u,v)=
\sum_{q}\Delta_qu \widetilde \Delta_{q}v,
\end{equation*}
$$
\textnormal{with}\quad {\widetilde \Delta}_{q}=\Delta_{q-1}+\Delta_{q}+\Delta_{q+1}.
$$
In usual, $T_{u}v$ is called para-product of $v$ by $u$ and  $\mathcal{R}(u,v)$ denotes the remainder term.
In addition, it is worthwhile to point out that ${\widetilde \Delta}_{q}\Delta_q=\Delta_q$ for $q\geq0$ by using the property of support of $\varphi$.
 \begin{definition}\label{def2.2}
For $s\in \mathbb{R}$, $(p,q)\in [1,+\infty]^{2}$ and $u\in \mathcal{S}'(\mathbb{R}^{n})$, the \emph{inhomogeneous Besov spaces} are defined by
\begin{equation*}
   {B}^{s}_{p,q}(\mathbb{R}^{n}):=\big\{u\in\mathcal{S}'(\mathbb{R}^{n})\big|\,\norm{u}_{{B}^{s}_{p,q}(\mathbb{R}^{n})}<\infty\big\}.
\end{equation*}
Here
\begin{equation*}
    \norm{u}_{{B}^{s}_{p,q}(\mathbb{R}^{n})}:=\begin{cases}
    \Big(\sum_{j\geq-1}2^{jsq}\norm{{\Delta}_{j}u}_{L^{p}(\mathbb{R}^{n})}^{q}\Big)^{\frac{1}{q}}\quad&\text{if}\quad r<\infty,\\
    \sup_{j\geq-1}2^{js}\norm{{\Delta}_{j}u}_{L^{p}(\mathbb{R}^{n})}\quad&\text{if}\quad r=\infty.
    \end{cases}
    \end{equation*}
\end{definition}
Since the dissipation only occurs in the horizontal direction to Equations \eqref{bs}, we need introduce  the following anisotropic space.
\begin{definition}\label{def-anisotropic}
For $s,t\in \mathbb{R}$, $(p,q)\in [1,+\infty]^{2}$ and $u\in \mathcal{S}'(\mathbb{R}^{3})$,  we define the \emph{anisotropic Besov spaces} as
\begin{equation*}
   {B}^{s,t}_{p,q}(\mathbb{R}^{3}):=\big\{u\in\mathcal{S}'(\mathbb{R}^{3})\big|\,\norm{u}_{{B}^{s,t}_{p,q}(\mathbb{R}^{3})}<\infty\big\},
\end{equation*}
where
\begin{equation*}
    \norm{u}_{{B}^{s,t}_{p,q}(\mathbb{R}^{3})}:=\begin{cases}
  \Big(\sum_{j,\,k\geq-1}2^{jsq}2^{ktq}\big\|{\Delta}^{h}_{j}{\Delta}^{v}_{k}u\big\|_{L^{p}(\mathbb{R}^{3})}^{q}\Big)^{\frac{1}{q}}
    \quad&\text{if}\quad r<\infty,\\
     =\sup_{j,\,k\geq-1}2^{js}2^{kt}\big\|{\Delta}^{h}_{j}{\Delta}^{v}_{k}u\big\|_{L^{p}(\mathbb{R}^{3})}\quad&\text{if}\quad r=\infty.
    \end{cases}
\end{equation*}
Here  and in what follows, $$\Delta_i^hf(x_h):=2^{2i}\int_{\RR^2}\varphi(x_h-2^i y)f(y)\intd y$$ and $$\Delta_j^vf(z):=2^{2j}\int_{\RR}\varphi(z-2^j y)f(y)\intd y.$$
\end{definition}
In the following, we briefly review some basic properties for $H^{s,t}$ spaces which will be useful later.
\begin{lemma}[\cite{MZ2012}]\label{properties}
\begin{enumerate}[\rm (i)]
There hold that
\item\label{eq.item-1}
For $s_{2}\geq s_{1}$ and $t_{2}\geq t_{1}$, one has $\norm{u}_{H^{s_{2},t_{2}}(\mathbb{R}^{3})}\hookrightarrow\norm{u}_{H^{s_{1},t_{1}}(\mathbb{R}^{3})}.$
    \item\label{eq.item-2} For $s_{1},s_{2},t_{1},t_{2}\in \mathbb{R}$, there exists $\theta\in[0,1]$ such that
    $$\norm{u}_{H^{\theta s_{1}+(1-\theta)s_{2},\theta t_{1}+(1-\theta)t_{2}}(\mathbb{R}^{3})}\leq
    \norm{u}^{\theta}_{H^{s_{2},t_{2}}(\mathbb{R}^{3})}\norm{u}^{1-\theta}_{H^{s_{1},t_{1}}(\mathbb{R}^{3})}.
    $$
  \item\label{eq.item-3}  For $s,t\geq 0$, $\norm{u}_{H^{s,t}(\mathbb{R}^{3})}$ is equivalent to $$\norm{u}_{L^{2}(\mathbb{R}^{3})}+\norm{\Lambda^{s}_{h}u}_{L^{2}(\mathbb{R}^{3})}
  +\norm{\Lambda^{t}_{v}u}_{L^{2}(\mathbb{R}^{3})}+\norm{\Lambda^{t}_{v}\Lambda^{s}_{h}u}_{L^{2}(\mathbb{R}^{3})}.$$
  \item \label{eq.item-4}$\norm{u}_{H^{s,t}(\mathbb{R}^{3})}\simeq\big\|\|u\|_{{H}^{s}(\mathbb{R}_{h}^{2})}\big\|_{{H}^{t}(\mathbb{R}_{v})}\simeq
  \big\|\|u\|_{{H}^{t}(\mathbb{R}_{v})}\big\|_{{H}^{s}(\mathbb{R}_{h}^{2})}.$
  \item\label{eq.item-5} For $s>1$ and $t>\frac 12$, $\norm{u}_{H^{s,t}(\mathbb{R}^{3})}$ is an algebra.
\end{enumerate}
\end{lemma}
Let us point out that the usual Sobolev spaces $H^s$ and $H^{s,t}$ coincide with  Besov spaces  $B_{2,2}^s$ and $B_{2,2}^{s,t}$, respectively.
\begin{lemma}[Morse estimate]\label{lem2.2}
Let $s>0$, $q\in[1,\infty]$. Then there exists a constant $C$ such that
\begin{equation}\|fg\|_{{
B}_{p,q}^s}\leq C\big(\| f\|_{L^{p_1}}\|g\|_{{ B}_{p_2,q}^s}+\|
g\|_{L^{r_1}}\|f\|_{{ B}_{r_2,q}^s}\big),
\end{equation}  where
$p_1,r_1\in[1,\infty]$ satisfy $\frac{1}{p}=\frac{1}{p_1}+\frac{1}{p_2}=\frac{1}{r_1}+\frac{1}{r_2}$.
\end{lemma}
\begin{proof}
The proof of Lemma \ref{lem2.2} is standard, here we omit the details. One also refer to \cite{CWZ2012} for the proof.
\end{proof}
\begin{lemma}\label{lem-p}
For any $p\in]1,\infty[$, there holds that
\begin{equation}
\|\nabla u\|_{L^p(D)}\leq C\frac{p^2}{p-1}\|\omega\|_{L^p(D)},
\end{equation}
where $C$ depending only on the domain $D$, and not on $p$.
\end{lemma}
\begin{proof}
The classical result of Calderon-Zygmund \cite{C-Z52} together with the Biot-Savart law allows to obtain the desired result for simple domains such as the whole space $\RR^n$, the half space or the ball.
As for the general case the proof is based on a rather awkward technique, developed in \cite{Yu,Y-66}.
\end{proof}
\begin{proposition}[\cite{che1}]\label{heat}
Let $u$ be solution of the classical heat equations
\begin{equation*}
\left\{\begin{array}{ll}
\partial_t u-\Delta u=0,\quad(t, x)\in \mathbb{R}^+\times\mathbb{R}^{n},\\
u|_{t=0}=u_{0}.
\end{array}\right.
\end{equation*}
Then there exists a constant $C>0$ such that  for each $j\geq0$,
 \begin{equation*}
\|\Delta_{j}u (t) \|_{L^{p}(\mathbb{R}^{n})}=\|e^{t\Delta }\Delta_{j}u_{0}\|_{L^{p}(\mathbb{R}^{n})}\leq Ce^{-ct2^{2j}}\|u_{0}\|_{L^{p}(\mathbb{R}^{n})}.
 \end{equation*}
\end{proposition}
\begin{lemma}\label{lem-anstro}
Let $s_i>1$ and $t_i>\frac12$ with $i=1,2$. Then there exists a constant $C$ such that
\begin{equation}\label{eq-dd}
\|u\|_{L^\infty(\RR^3)}\leq C\Big(\|u\|_{L^2(\RR^3)}+\|\Lambda_{h}^{s_1}u\|_{L^2(\RR^3)}+\|\Lambda_{v}^{t_1}u\|_{L^2(\RR^3)}+
\|\Lambda_h^{s_2}\Lambda_v^{t_2}u\|_{L^2(\RR^3)}\Big).
\end{equation}
\end{lemma}
\begin{proof}
Thanks to Littewood-Paley decomposition, one has
\begin{align*}
\|u\|_{L^\infty(\RR^3)}\leq&\sum_{i,j\geq-1}\|\Delta_i^h\Delta_j^vu\|_{L^\infty(\RR^3)}\nonumber\\
=&\|\Delta_{-1}^h\Delta_{-1}^vu\|_{L^\infty(\RR^3)}
+\sum_{i\geq0}\|\Delta_i^h\Delta_{-1}^vu\|_{L^\infty(\RR^3)}+\sum_{j\geq0}|\|\Delta_{-1}^h\Delta_j^vu\|_{L^\infty(\RR^3)}\nonumber\\
&+\sum_{i,j\geq0}\|\Delta_i^h\Delta_j^vu\|_{L^\infty(\RR^3)}\nonumber\\
:=&I_1+I_2+I_3+I_4.
\end{align*}
It is clear that
\begin{equation*}
I_1\leq C\|u\|_{L^2}.
\end{equation*}
By using the Bernstein inequality, the Minkowski inequality and the fact that $s_1>1$, we infer that
\begin{align*}
I_2\leq&C\sum_{i\geq 0}2^i\|\Delta_i^h\Delta_{-1}^vu\|_{L^2(\RR^3)}\nonumber\\
\leq&C\sum_{i\geq 0}2^{i(1-s_1)}\|\Delta_i^h\Lambda_h^{s_1}u\|_{L^2(\RR^3)}\nonumber\\
\leq&C\|\Lambda_h^{s_1}u\|_{L^2(\RR^3)}.
\end{align*}
And similarly, we can conclude that for $t_1>\frac12$,
\begin{equation*}
I_3\leq C\|\Lambda_v^{t_1}u\|_{L^2(\RR^3)}.
\end{equation*}
As for the term $I_4$, the Bernstein inequality and  the Minkowski inequality enable us to conclude that for $s_2>1$ and $t_2>\frac12$,
\begin{align*}
I_4\leq&C\sum_{i,j\geq 0}2^i2^{\frac12j}\|\Delta_i^h\Delta_{j}^vu\|_{L^2(\RR^3)}\nonumber\\
\leq&C\sum_{i,j\geq 0}2^{i(1-s_2)}2^{j(\frac12-t_2)}\|\Delta_i^h\Delta_{j}^v\Lambda_h^{s_2}\Lambda_v^{t_2}u\|_{L^2(\RR^3)}\nonumber\\
\leq&C\|\Lambda_h^{s_2}\Lambda_v^{t_2}u\|_{L^2(\RR^3)}.
\end{align*}
Collecting these estimates yields the desired result \eqref{eq-dd}.
\end{proof}
Next, we will give a new type space-time logarithmic inequality which is a key component in our analysis, by using the low-high frequency decomposition technique.
\begin{lemma}\label{log}
Let $a\in]0,1]$, $p,q\in[1,\infty]$ and $s>\frac{n}{p}$. Assume that $f\in L^{1}_{T}(B_{p,q}^{s}(\mathbb{R}^n))$ such that
\begin{equation*}
\sup_{2\leq j<\infty}\int^{T}_{0}\frac{\norm{S_{j}f(t)}_{L^{\infty}(\mathbb{R}^n)}}{j^{a}}\mathrm{d}t\leq\infty.
\end{equation*}
Then the following inequality holds:
\begin{equation}\label{LOG-E}
\int^{T}_{0}\norm{f(t)}_{L^\infty(\mathbb{R}^n)}\mathrm{d}t\leq C\Big(1+\sup_{2\leq j<\infty}\int^{T}_{0}\frac{\norm{S_{j}f(t)}_{L^{\infty}(\mathbb{R}^n)}}{j^{a}}\mathrm{d}t\Big(\log\big(e+\norm{f}_{L^{1}_{T}(B_{p,q}^{s}(\mathbb{R}^n))}\big)\Big)^{a}\Big).
\end{equation}
Here the constant $C$ independent of $f$ and $T$.
\end{lemma}
\begin{proof}
According to the Littlewood-Paley decomposition, we decompose $f$ into two parts as follows:
\begin{equation*}
f=S_{N}f+\sum_{k\geq N}\Delta_{k}f,
\end{equation*}
where $N$ is a positive integer to be fixed later.

For the low-frequency part $S_{N}f$, it is clear that for $N\geq2$
\begin{equation*}
\int^{T}_{0}\norm{S_{N}f(t)}_{L^{\infty}(\mathbb{R}^n)}\mathrm{d}t\leq N^{a}\sup_{2\leq j<\infty}\int^{T}_{0}\frac{\norm{S_{j}f(t)}_{L^{\infty}(\mathbb{R}^n)}}{j^{a}}\mathrm{d}t.
\end{equation*}
Next, for the high-frequency part, in view of the Bernstein inequality and  the definition of Besov space, we have
\begin{align}
\int^{T}_{0}\Big\|\sum_{k\geq N}\Delta_{k}f(t)\Big\|_{L^{\infty}(\mathbb{R}^n)}\mathrm{d}t\leq&C\int^{T}_{0}\sum_{j\geq N}2^{j(\frac{n}{p}-s)}2^{js}\norm{\Delta_{k}f(t)}_{L^{p}(\mathbb{R}^n)}\mathrm{d}t\nonumber\\
\leq&C2^{-N(s-\frac np)}\int^{T}_{0}\norm{f(t)}_{B_{p,q}^{s}(\mathbb{R}^n)}\mathrm{d}t\nonumber,
\end{align}
in the last line we have used the H\"older inequality.

Collecting these estimates, we thus get
\begin{equation}\label{LOG-1}
\int^{T}_{0}\norm{f(t)}_{L^{\infty}(\mathbb{R}^n)}\mathrm{d}t\leq N^{a}\sup_{2\leq j<\infty}\int^{T}_{0}\frac{\norm{S_{j}f(t)}_{L^{\infty}(\mathbb{R}^n)}}{j^{a}}\mathrm{d}t
+C2^{-N(s-\frac np)}\int^{T}_{0}\norm{f(t)}_{B_{p,q}^{s}(\mathbb{R}^n)}\mathrm{d}t.
\end{equation}
Now we choose $N$ which more than 2 such that $2^{-N(s-\frac np)}\int^{T}_{0}\|f(t)\|_{B_{p,q}^{s}(\mathbb{R}^n)}\mathrm{d}t\leq 1$, i.e.,
\begin{equation*}
N\geq\max\Big\{2,\,\frac{\log\|f\|_{L_{T}^{1}(B_{p,q}^{s}(\mathbb{R}^n))}}{(s-\frac np)\log2}\Big\}.
\end{equation*}
This together with \eqref{LOG-1} yields the desired result \eqref{LOG-E}.
\end{proof}
\subsection{Study of the flow map}
In this subsection, we first review some basic results about the flow in \cite{A-H-K0}.
And we give another form of a growth  estimate of $\frac\rho r$ which be suitable to our problems according  to the  generalized flow map associated to  an axisymmetric vector field.
\begin{equation}\label{eq.flow}
\psi(t,s,x)=x+\int_s^tu(\tau,\psi(\tau,s,x))\mathrm{d}\tau.
\end{equation}
It is well-known that if the vector field
$u$  lies in  $L^1_{\text{loc}}(\RR;\, {\rm Lipschitz})$ then the generalized flow is
uniquely determined and exists globally in time. In addition, the incompressible condition that $\text{div}\, u=0$ guarantees that for every $t,s\in\RR$, $\psi(t,s)$ is a diffeomorphism that preserves Lebesgue measure and
\begin{equation*}
\psi^{-1}(t,s,x)=\psi(s,t,x).
\end{equation*}
Now we begin to show a new form of estimate for the quantity $\big\|\frac{\rho(t)}{r}\big\|_{L^2}$ which different from the quadratic growth estimate \eqref{growth-2} in \cite{A-H-K0}. This is the  cornerstone for  establishing of the quantity  $\|\omega(t)\|_{L^2}$.
\begin{proposition}\label{cor1}
Let $u$ be a smooth axisymmetric vector field  with zero divergence and $\rho$ be a solution of the transport equation
\begin{equation*}
\partial_{t}\rho+u\cdot\nabla \rho=0,\quad
\rho_{| t=0}=\rho_{0}.
\end{equation*}
Assume in addition that
\begin{equation*}
\rho_0\in L^2\cap L^\infty,\quad\mathrm{d}\big(\textnormal{Supp }\rho_0,(Oz)\big):=r_0>0 \quad\hbox{and}\quad
\textnormal{diam}\big(\Pi_z(\textnormal{Supp }\rho_0)\big):=\mathrm{d}_0<\infty.
\end{equation*}
 Then we have
\begin{equation}\label{density-e}
\Big\|\frac{\rho}{r}(t)\Big\|_{L^2}^2
\leq\frac{1}{r_0^2}\|\rho_0\|_{L^2}^2
+2\pi\|\rho_0\|_{L^\infty}^2\int_0^t\Big\|\frac{u_r}{r}(\tau)\Big\|_{L^\infty}\mathrm{d}\tau
\Big(\mathrm{d}_0+2\int_0^t\norm{\nabla_h u(\tau)}^{\frac{1}{2}}_{L^2}\norm{\nabla_h \omega(\tau)}^{\frac{1}{2}}_{L^{2}}\mathrm{d}\tau\Big),
\end{equation}
where $r=(x_1^2+x_2^2)^{\frac12}$.
\end{proposition}
\begin{proof}
We have from the definition that
\begin{align}\label{refine-1}
\Big\|\frac{\rho}{r}(t)\Big\|_{L^2}^2
=&\int_{r\geq r_0}\frac{{\rho^2(t,x)}}{r^2}\mathrm{d}x+\int_{r\le r_0}\frac{{\rho^2(t,x)}}{r^2}\mathrm{d}x\nonumber\\
\leq&
\frac{1}{r_0^2}\|\rho(t)\|_{L^2}^2+\|\rho(t)\|_{L^\infty}^2\int_{\{r\le r_0\}\cap
\textnormal{Supp }\rho(t)}\frac{1}{r^2}\mathrm{d}x\nonumber\\
\leq&\frac{1}{r_0^2}\|\rho_0\|_{L^2}^2+\|\rho_0\|_{L^\infty}^2\int_{\{r\le r_0\}\cap
\textnormal{Supp }\rho(t)}\frac{1}{r^2}\mathrm{d}x.
\end{align}
On the other hand, according to  \cite[Proposition 3.2]{A-H-K0}, we know
\begin{equation*}
\text{diam}(\Pi_{z}\text{Supp}\rho(t))\leq\text{diam}(\Pi_{z}\text{Supp}\rho_0)+2\int_0^t\norm{u_z(\tau)}_{L^\infty}\mathrm{d}\tau.
\end{equation*}
This together with the maximum principle of $\rho$ gives that
\begin{align*}
\int_{\{r\le r_0\}\cap \textnormal{Supp }\rho(t)}\frac{1}{r^2}\mathrm{d}x
&\le2\pi
\Big(\int_{r_0e^{-\int_0^t\|\frac{u_r}{r}(\tau)\|_{L^\infty}\mathrm{d}\tau}
\le r\le r_0}\frac{1}{r}\mathrm{d}r\Big)\Big(\int_{\Pi_z(\textnormal{Supp } \rho(t))}\mathrm{d}z\Big)\\
&\le2\pi\int_0^t\Big\|\frac{u_r}{r}(\tau)\Big\|_{L^\infty}\mathrm{d}\tau\Big(\mathrm{d}_0+2\int_0^t\|u_{z}(\tau)\|_{L^\infty}\mathrm{d}\tau\Big).
\end{align*}
The divergence free condition  guarantees that $\norm{\Delta u_{z}}_{L^{2}}\leq C\norm{\nabla\nabla_h u}_{L^{2}}$ and $\norm{\nabla u_{z}}_{L^{2}}\leq C\norm{\nabla_h u}_{L^{2}}$. Thus, we can deduce by using the interpolation theorem that
\begin{equation*}
\begin{split}
\norm{u_z}_{L^{\infty}}\leq \norm{u_{z}}^{\frac{1}{2}}_{L^{6}}\norm{\Delta u_{z}}^{\frac{1}{2}}_{L^{2}}
\leq C\norm{\nabla u_z}^{\frac{1}{2}}_{L^2}\norm{\nabla\nabla_h u}^{\frac{1}{2}}_{L^{2}}\leq C\norm{\nabla_h u}^{\frac{1}{2}}_{L^2}\norm{\nabla_h \omega}^{\frac{1}{2}}_{L^{2}}.
\end{split}
\end{equation*}
Inserting this inequality in \eqref{refine-1}, we eventually obtain the desired result \eqref{density-e}.
\end{proof}
\section{A priori estimates}\label{Sec-3}
\setcounter{section}{3}\setcounter{equation}{0}
This section is devoted to a priori estimates which can be viewed as a preparation for the proof of  \mbox{Theorem \ref{global}}.
\subsection{Weak a priori estimates}
\begin{proposition}\label{energy}
Let $u_0\in L^2$ be a  vector field  with zero divergence and $\rho_0\in L^2\cap L^\infty.$
Then every smooth solution of \eqref{bs} satisfies
$$
\|\rho(t)\|_{L^p}\leq\|\rho_0\|_{L^p}\quad \text{for}\quad p\in[2,\infty].
$$
$$
\|u(t)\|_{L^2}^2+2\int_0^t\|\nabla_{h} u(\tau)\|_{L^2}^2\mathrm{d}\tau\leq 2\big(\|u_0\|^2_{L^2}+t^2\|\rho_0\|^{2}_{L^2}\big).
$$
\end{proposition}
\begin{proof}
According to \eqref{eq.flow}, one can write
 \begin{equation*}
 \rho(t,x)=\rho_0\big(\psi^{-1}(t,x)\big),
 \end{equation*}
This together with the incompressible condition implies the first estimate.

Next, we take the $L^2$-inner product of  the velocity equation with $u$.
Integrating by parts with respect to space leads to
\begin{equation}\label{eqs1}
\frac12\frac{\mathrm{d}}{\mathrm{d}t}\|u(t)\|_{L^2}^2+\|\nabla_{h} u(t)\|_{L^2}^2\le\|u(t)\|_{L^2}\|\rho(t)\|_{L^2}.
\end{equation}
This means that
$$
\frac{\mathrm{d}}{\mathrm{d}t}\|u(t)\|_{L^2}\le\|\rho(t)\|_{L^2}.
$$
Integrating in time this inequality yields
$$
\|u(t)\|_{L^2}\le\|u_0\|_{L^2}+\int_0^t\|\rho(\tau)\|_{L^2}\mathrm{d}\tau.
$$
Since  $\|\rho(t)\|_{L^2}=\|\rho_0\|_{L^2},$ then
$$
\|u(t)\|_{L^2}\le\|u_0\|_{L^2}+t\|\rho_0\|_{L^2}.
$$
Putting this estimate into \eqref{eqs1} gives
\begin{align*}
\|u(t)\|_{L^2}^2+2\int_0^t\|\nabla_{h} u(\tau)\|_{L^2}^2\mathrm{d}\tau
\leq&\|u_0\|_{L^2}^2+2\int_0^t\big(\|u_0\|_{L^2}+\tau\|\rho_0\|_{L^2}\big)\|\rho_0\|_{L^2}\intd\tau\nonumber\\
\leq&\|u_0\|_{L^2}^2+2\big(\|u_0\|_{L^2}+\frac12t\|\rho_0\|_{L^2}\big)\|\rho_0\|_{L^2}t.
\end{align*}
This gives the second desired estimate.
\end{proof}
Next, we intend to review the behavior of the operator $\frac{\partial_r}{r}\Delta^{-1}$ over axisymmetric functions and the the algebraic relation between $\frac{u_r}{r}$ and $\frac{\omega_\theta}{r}$.
\begin{lemma}[\cite{hrou}]\label{prop1}
Assume that $u$ is an axisymmetric smooth scalar function, then there holds that
\begin{equation}\label{prop1-12}
\Big(\frac{\partial_r}{r}\Big)\Delta^{-1}u(x)=\frac{x_2^2}{r^2}\mathcal{R}_{11}u(x)+\frac{x_1^2}{r^2}\mathcal{R}_{22}u(x)-2\frac{x_1x_2}{r^2}\mathcal{R}_{12}u(x),
\end{equation}
where $\mathcal{R}_{ij}=\partial_{ij}\Delta^{-1}$.
\end{lemma}
\begin{lemma}\label{prop-identity}
Assume that $u$ be an axisymmetric vector-field without swirl  satisfying $\mathrm{div}\, u=0$ and $\omega=\text{\rm curl}\,u$. Then
\begin{equation}\label{identity}
\frac{u_{r}}{r}=\partial_{z}\Delta^{-1}\left(\frac{\omega_\theta}{r}\right)-2\frac{\partial_r}{r}\Delta^{-1}\partial_z\Delta^{-1}\left(\frac{\omega_\theta}{r}\right)
\end{equation}
Furthermore, one has
\begin{equation}\label{point-ansitro}
\Big|\frac{u_{r}}{r}\Big|\leq \Big|\partial_{z}\Delta^{-1}\Big(\frac{\omega_\theta}{r}\Big)\Big|+\sum^{2}_{i,j=1}\Big|\mathcal{R}_{ij}\partial_{z}\Delta^{-1}\Big(\frac{\omega_\theta}{r}\Big)\Big|,
\end{equation}
and
\begin{equation}\label{eq-ds}
\Big\|\partial^2_z\Big(\frac{u_r}{r}\Big)\Big\|_{L^p}\leq C\Big\|\partial_z\Big(\frac{\omega_\theta}{r}\Big)\Big\|_{L^p},\quad\text{for}\quad p\in]1,\infty[.
\end{equation}
\end{lemma}
\begin{proof}
One can refer to \cite{MZ2012} for the proof of \eqref{identity} and \eqref{point-ansitro}. Next, we turn to show \eqref{eq-ds}. According to \eqref{identity}, we obtain
\begin{equation}
\partial_z^2\frac{u_{r}}{r}=\partial_{z}\Delta^{-1}\partial_z^2\left(\frac{\omega_\theta}{r}\right)
-2\frac{\partial_r}{r}\Delta^{-1}\partial_z\Delta^{-1}\partial_z^2\left(\frac{\omega_\theta}{r}\right).
\end{equation}
Using  Lemma \ref{prop1-12} and applying the $L^p$-boundedness of Riesz operator, we eventually obtain \eqref{eq-ds}.
\end{proof}
The next proposition describes some  estimates linking  the velocity to the
vorticity in virtue of the Biot-Savart law and Lemma \ref{prop-identity}.
\begin{proposition}[\cite{MZ2012}]\label{biot-s}
Assume that $u$ is an axisymmetric vector-field without swirl  satisfying $\mathrm{div}\, u=0$ and vorticity
$\omega=\omega_\theta e_\theta$. Then
\begin{equation}\label{bs-1}
\|u\|_{L^\infty(\mathbb{R}^3)}\le C\|\omega_\theta\|_{L^{2}(\mathbb{R}^3)}^{\frac12}\|\nabla_h\omega_\theta\|_{ L^2(\mathbb{R}^3)}^{\frac12}.
\end{equation}
and
\begin{equation}\label{bs-2}
\Big\|\frac{u_r}{ r}\Big\|_{L^\infty(\mathbb{R}^3)}\le
 C\Big\|\frac{\omega_\theta}{r}\Big\|_{L^{2}(\mathbb{R}^3)}^{\frac12}\Big\|\nabla_h\Big(\frac{\omega_\theta}{r}\Big)\Big\|_{L^{2}(\mathbb{R}^3)}^{\frac12}.
\end{equation}
\end{proposition}

\subsection{Strong a priori estimates}
In this subsection, our task is to obtain the  global Lipschitz  estimates of the vector field. Now, let us begin with the estimate of $\|u(t)\|_{H^1}$.

\begin{proposition}\label{prop-cruc}
Let $u_0\in H^1$  be an axisymmetric vector field with zero
divergence such that $\frac{{\omega_0}}{r}\in L^2.$ Let
$\rho_0\in L^2\cap L^\infty$ depending only on $(r,z)$ such that
$\hbox{\rm Supp }\rho_0$ does not intersect the axis $(Oz)$ and
$\Pi_z(\hbox{\rm Supp }\rho_0)$ is a compact set. Then every smooth
solution $(u,\rho)$ of the system \eqref{bs} satisfies for every
$t\geq 0$,
\begin{equation*}
\|\omega_\theta(t)\|_{L^2}^2+\int_0^t\Big(\|\nabla_h\omega_\theta(\tau)\|_{H^1}^2\mathrm{d}\tau+\big\|{\omega_\theta\over r}(\tau)\big\|_{L^2}^2\Big)\intd\tau+\Big\|\frac{\omega_\theta}{r}(t)\Big\|_{L^2}^2+\int_0^t\Big\|\nabla_h\Big(\frac{\omega_\theta}{r}\Big)(\tau)\Big\|^2_{L^{2}}\intd\tau\le Ce^{\exp{Ct^{17}}},
\end{equation*}
where the constant $C$ depends on the initial data.
\end{proposition}
\begin{proof}
Taking the $L^2$-inner product of  the equation \eqref{tourbillon} with $\omega_\theta$ we get
\begin{align}\label{eq.wdu}
\frac12\frac{\mathrm{d}}{\mathrm{d}t}\|\omega_\theta(t)\|_{L^2}^2+\|\nabla_{h}\omega_\theta(t)\|_{L^2}^2
+\Big\|{\omega_\theta\over r}(t)\Big\|_{L^2}^2
&=\int_{\mathbb{R}^3} u_r{\omega_\theta\over r}\omega_\theta \mathrm{d}x
-\int_{\mathbb{R}^3}\partial_r\rho\,\omega_\theta \mathrm{d}x.
\end{align}
For the first integral term in the right side of \eqref{eq.wdu}, by using Lemma \ref{lema.2} and the Young inequality, we get
\begin{equation*}
\begin{split}
\int_{\mathbb{R}^3} u_r{\omega_\theta\over r}\omega_\theta \mathrm{d}x\le&\|u_{r}\|^{\frac{1}{2}}_{L^2}\|\partial_zu_{r}\|^{\frac{1}{2}}_{L^2}\norm{\omega_\theta\over r}^{\frac{1}{2}}_{L^2}\norm{\nabla_h\Big(\frac{\omega_\theta}{ r}\Big)}^{\frac{1}{2}}_{L^2}
\|\omega_\theta\|^{\frac{1}{2}}_{L^2}\|\nabla_h\omega_\theta\|^{\frac{1}{2}}_{L^2}\\
\le&\frac12\|u\|_{L^2}\|\omega_{\theta}\|^{2}_{L^2}+\frac14\norm{\omega_\theta\over r}^{2}_{L^2}\norm{\nabla_h\Big(\frac{\omega_\theta}{ r}\Big)}^{2}_{L^2}
+\frac{1}{4}\|\nabla_h\omega_\theta\|^{2}_{L^2}.
\end{split}
\end{equation*}
For the second integral term in the right side of \eqref{eq.wdu}, by the H\"older inequality, we obtain
\begin{equation*}
\begin{split}
-\int_{\mathbb{R}^3}\partial_r\rho\,\omega_\theta \intd x\le&\|\rho\|_{L^2}\Big(\Big\|{\omega_\theta\over r}\Big\|_{L^2}
+\|\partial_r\omega_\theta\|_{L^2}\Big)\\
\leq&\|\rho_{0}\|^{2}_{L^2}+\frac{1}{4}\Big(\Big\|{\omega_\theta\over r}\Big\|^{2}_{L^2}
+\|\partial_r\omega_\theta\|^{2}_{L^2}\Big).
\end{split}
\end{equation*}
Indeed, integration by parts gives
\begin{equation*}
\begin{split}
-\int_{\mathbb{R}^3}\partial_r\rho\,\omega_\theta \mathrm{d}x=&-2\pi\int\partial_r\rho\,\omega_\theta r\mathrm{d}r\,\mathrm{d}z\\
=&2\pi\int\rho\partial_r\omega_\theta\, r\mathrm{d}r\,\mathrm{d}z+2\pi\int\rho{\omega_\theta\over r} r\mathrm{d}r\,\mathrm{d}z\\
=&\int_{\mathbb{R}^3}\rho\Big(\partial_r\omega_\theta+{\omega_\theta\over r}\Big) \mathrm{d}x.
\end{split}
\end{equation*}
Putting together these estimates and using the fact that $\|\nabla_{h}\omega_\theta\|_{L^2}^2=\|\partial_r\omega_\theta\|_{L^2}^2+\|\omega_\theta/r\|_{L^{2}}$ yield
\begin{equation*}
\begin{split}
&\frac{\mathrm{d}}{\mathrm{d}t}\|\omega_\theta(t)\|_{L^2}^2+\|\nabla_{h}\omega_\theta(t)\|_{L^2}^2+\big\|{\omega_\theta\over r}(t)\big\|_{L^2}^2\\
\leq&\|u(t)\|_{L^2}\|\omega_{\theta}(t)\|^{2}_{L^2}+\frac12\norm{\frac{\omega_\theta}{r}(t) }^{2}_{L^2}\norm{\nabla_h\Big(\frac{\omega_\theta}{ r}\Big)(t)}^{2}_{L^2}
+2\|\rho_{0}\|^{2}_{L^2}.
\end{split}
\end{equation*}
Next, integrating the above inequality with respect to time,  we readily get
\begin{align}\label{225}
&\|\omega_\theta(t)\|_{L^2}^2+\int_0^t\|\nabla_{h}\omega_\theta(\tau)\|_{L^2}^2\intd\tau+\int_0^t\big\|{\omega_\theta\over r}(\tau)\big\|_{L^2}^2\intd\tau\nonumber\\
\leq&\|\omega_\theta(0)\|^{2}_{L^2}+\norm{u}_{L_{t}^{\infty}L^{2}}\int_{0}^{t}\|\omega_{\theta}(\tau)\|^{2}_{L^2}\mathrm{d}\tau+\frac12\norm{\omega_\theta\over r}^{2}_{L^\infty_{t}L^2}\int_0^t\norm{\nabla_h\Big(\frac{\omega_\theta}{ r}\Big)(\tau)}^{2}_{L^2}\intd\tau
+2t\|\rho_{0}\|^{2}_{L^2}\nonumber\\
\leq&\|\omega_\theta(0)\|^{2}_{L^2}+\norm{u_0}_{L^{2}}\int_{0}^{t}\|\omega_{\theta}(\tau)\|^{2}_{L^2}\mathrm{d}\tau+\frac12\norm{\omega_\theta\over r}^{2}_{L^\infty_{t}L^2}\int_0^t\norm{\nabla_h\Big(\frac{\omega_\theta}{ r}\Big)(\tau)}^{2}_{L^2}\intd\tau
+2t\|\rho_{0}\|^{2}_{L^2}.
\end{align}
To show the estimate for the \mbox{quantity $\Gamma:=\frac{\omega_\theta}{r}$.} We observe that the quantity $\Gamma$ satisfies the following equation
\begin{equation}\label{equation_i}
\big(\partial_t+u\cdot\nabla\big)\Gamma-\big(\Delta_{h}+{{2 \over r}}\partial_r\big) \Gamma =-\frac{\partial_r\rho}{r}.
\end{equation}
Taking the $L^2$-inner product of  \eqref{equation_i} with $\Gamma$  and integrating by parts, we get
\begin{align*}
\frac12\frac{\mathrm{d}}{\mathrm{d}t}\|\Gamma(t)\|_{L^2}^2
+\|\nabla_h\Gamma(t)\|_{L^2}^2
-4\pi\int\partial_r(\Gamma)
{\Gamma} \mathrm{d}r\mathrm{d}z
&=
-2\pi\int\partial_r\rho\,\Gamma \mathrm{d}r\mathrm{d}z.
\end{align*}

For the term in the right side of equality above, integrating by parts and using the fact $\rho(t,0,z)=0$, we obtain
\begin{equation*}
\begin{split}
-2\pi\int\partial_r\rho\,\Gamma \mathrm{d}r\mathrm{d}z
=2\pi\int{\rho\over r}\,\partial_r\Gamma\, r \mathrm{d}r\mathrm{d}z=\int_{\mathbb{R}^3}\frac{\rho}{r}\partial_r\Gamma \mathrm{d}x
\le
\Big\|{\rho\over r}\Big\|_{L^2}\| \partial_r\Gamma\|_{L^2}\leq \frac12\Big\|\frac{\rho}{r}\Big\|^{2}_{L^{2}}+\frac12\norm{\partial_r\Gamma}^{2}_{L^{2}}.
\end{split}
\end{equation*}
Since
$$
4\pi\int\partial_r(\Gamma)\Gamma \mathrm{d}r\mathrm{d}z
=2\pi\int_{\RR}\int_0^{+\infty}\partial_r(\Gamma)^2 \mathrm{d}r\mathrm{d}z
\le
0,
$$
then one can conclude that
\begin{equation*}
\frac{\mathrm{d}}{\mathrm{d}t}\Big\|{\omega_\theta\over r}(t)\Big\|_{L^2}^2
+\Big\|\nabla_{h}\Big({\omega_\theta\over r}\Big)(t)\Big\|_{L^2}^2
\le\Big\|{\rho\over r}(t)\Big\|_{L^2}^2.
\end{equation*}
Integrating  this inequality in time, we immediately get
\begin{equation}\label{eq08}
\Big\|{\omega_\theta\over r}(t)\Big\|_{L^2}^2
+\int_0^t\Big\|\nabla_{h}\Big({\omega_\theta\over r}\Big)(\tau)\Big\|_{L^2}^2\mathrm{d}\tau
\leq\Big\|{\omega_\theta(0)\over r}\Big\|_{L^2}^2+
\int_0^t\Big\|{\rho\over r}(\tau)\Big\|_{L^2}^2\mathrm{d}\tau.
\end{equation}
Plugging  \eqref{density-e} of \mbox{Proposition \ref{cor1}}
into \eqref{eq08}, we have
\begin{equation*}
\begin{split}
&\Big\|{\omega_\theta\over r}(t)\Big\|_{L^2}^2
+\int_0^t\Big\|\nabla_{h}\Big({\omega_\theta\over r}\Big)(\tau)\Big\|_{L^2}^2\intd\tau\\
\leq&\Big\|{\omega_\theta(0)\over r}\Big\|_{L^2}^2+
\frac{\|\rho_0\|^{2}_{L^2}}{r_0^2}t+2\pi\|\rho_0\|^{2}_{L^\infty}\intd_0t\int_0^t\Big\|\frac{u_r}{r}(\tau)\Big\|_{L^\infty}\mathrm{d}\tau
\\
&+4\pi\|\rho_0\|^{2}_{L^\infty} t\int_0^t\Big\|\frac{u_r}{r}(\tau)\Big\|_{L^\infty}\mathrm{d}\tau\,\int_0^{t}\|\nabla_{h} u(\tau)\|^{\frac{1}{2}}_{L^2}\|\nabla_{h} \omega(\tau)\|^{\frac{1}{2}}_{L^2}\mathrm{d}\tau\\
\le&
C(1+t)+Ct^{\frac32}\Big\|\frac{u_r}{r}\Big\|_{L^{2}_{t}L^\infty}
+C t^{2}\Big\|\frac{u_r}{r}\Big\|_{L^{2}_{t}L^\infty}\,\Big(\int_0^{t}\|\nabla_{h} u(\tau)\|^{2}_{L^2}\mathrm{d}\tau\Big)^{\frac{1}{4}}\Big(\int_0^{t}\|\nabla_{h} \omega(\tau)\|^2_{L^2}\mathrm{d}\tau\Big)^{\frac{1}{4}}.
\end{split}
\end{equation*}
It follows from Proposition \ref{energy} that,
\begin{align}\label{eq10}
&\big\|{\omega_\theta\over r}(t)\big\|_{L^2}^2
+\int_0^t\big\|\nabla_{h}\Big({\omega_\theta\over r}\Big)(\tau)\big\|_{L^2}^2\intd\tau \nonumber\\
\le& C (1+t^{4})\Big(1+\int_0^t\big\|\frac{u_r}{r}(\tau)\big\|_{L^\infty}^2\mathrm{d}\tau\Big)
+\frac14\Big(\int_0^{t}\|\nabla_{h} \omega(\tau)\|^2_{L^2}\mathrm{d}\tau\Big)^{\frac12}.
\end{align}
By using \eqref{bs-2} and the Young inequality, one can conclude that
\begin{align*}
C(1+t^{4})\int_0^{t}\Big\|\frac{u_r}{r}(\tau)\Big\|_{L^\infty}^2\mathrm{d}\tau
\leq&C(1+t^{4})\int_0^{t}\Big\|\frac{\omega_\theta}{r}(\tau)\Big\|^{\frac12}_{L^2}\Big\|\nabla_h\Big(\frac{\omega_\theta}{r}\Big)(\tau)\Big\|^{\frac12}_{L^2}\intd\tau\\
\leq&C(1+t^{8})\int_0^{t}\Big\|\frac{\omega_\theta}{r}(\tau)\Big\|_{L^2}^2\mathrm{d}\tau
+\frac12\int_0^{t}\big\|\nabla_{h}\Big(\frac{\omega_\theta}{r}\Big)(\tau)\big\|_{L^2}^2\mathrm{d}\tau.
\end{align*}
Plugging  this estimate into \eqref{eq10}, we obtain
\begin{align}\label{eq100}
\Big\|{\omega_\theta\over r}(t)\Big\|_{L^2}^2
+\int_0^t\big\|\nabla_{h}\Big({\omega_\theta\over r}\Big)(\tau)\big\|_{L^2}^2\intd\tau
\leq&\Big\|{\omega_\theta(0)\over r}\Big\|_{L^2}^2+
C(1+t^{8})\Big(1+\int_0^{t}\Big\|\frac{\omega_\theta}{r}(\tau)\Big\|_{L^2}^2\mathrm{d}\tau\Big)\nonumber\\
&+\frac14\Big(\int_0^{t}\|\nabla_{h} \omega(\tau)\|^2_{L^2}\mathrm{d}\tau\Big)^{\frac12}.
\end{align}
Note that
\begin{align*}
\frac12\norm{\omega_\theta\over r}^{2}_{L^\infty_{t}L^2}\int_0^t\norm{\nabla_h\Big(\frac{\omega_\theta}{ r}\Big)(\tau)}^{2}_{L^2}\intd\tau\leq&\frac18
\Big(\Big\|{\omega_\theta\over r}\Big\|_{L^\infty_tL^2}^2
+\int_0^t\big\|\nabla_{h}\Big({\omega_\theta\over r}\Big)(\tau)\big\|_{L^2}^2\intd\tau\Big)^{2}\\
\leq&\frac18\Big\|{\omega_\theta(0)\over r}\Big\|_{L^2}^4+
C(1+t^{16})\Big(1+\int_0^{t}\Big\|\frac{\omega_\theta}{r}(\tau)\Big\|_{L^2}^2\mathrm{d}\tau\Big)^{2}\nonumber\\
&+\frac{1}{64}\Big(\int_0^{t}\|\nabla_{h} \omega(\tau)\|^2_{L^2}\mathrm{d}\tau\Big).
\end{align*}
This together with \eqref{225} and \eqref{eq100} give
\begin{align*}
&\|\omega_\theta(t)\|_{L^2}^2+\int_0^t\Big(\|\nabla_{h}\omega_\theta(\tau)\|_{L^2}^2
+\Big\|{\omega_\theta\over r}(\tau)\Big\|_{L^2}^2\Big)\intd\tau+\big\|{\omega_\theta\over r}(t)\big\|_{L^2}^4
+\Big(\int_0^t\Big\|\nabla_h\Big({\frac{\omega_\theta}{r}\Big)(\tau)}\Big\|_{L^2}^2\intd\tau\Big)^2
\nonumber\\
\leq&C\norm{u}_{L_{t}^{\infty}L^{2}}\int_{0}^{t}\|\omega_{\theta}(\tau)\|^{2}_{L^2}\mathrm{d}\tau+C_0(1+t^{16})\Big(1+\int_0^{t}\big\|\frac{\omega_\theta}{r}(\tau)\big\|_{L^2}^2\mathrm{d}\tau
\Big)^{2}\nonumber\\&+\frac12\int_0^{t}\|\nabla_{h} \omega(\tau)\|^2_{L^2}\mathrm{d}\tau
+2t\|\rho_{0}\|^{2}_{L^2}+\big\|{\omega_\theta(0)\over r}\big\|_{L^2}^2+\frac18\big\|{\omega_\theta(0)\over r}\big\|_{L^2}^2.
\end{align*}
The Gronwall inequality  ensures that
\begin{equation*}
\|\omega_\theta(t)\|_{L^2}^2+\int_0^t\Big(\|\nabla_{h}\omega_\theta(\tau)\|_{L^2}^2
+\Big\|{\omega_\theta\over r}(\tau)\Big\|_{L^2}^2\Big)\intd\tau+\Big\|\frac{\omega_\theta}{r}(t)\Big\|_{L^2}^2
+\int_0^t\Big\|\nabla_{h}\Big(\frac{\omega_\theta}{r}\Big)(\tau)\Big\|_{L^2}^2\intd\tau
\leq C e^{\exp{C t^{17}}}.
\end{equation*}
This completes the proof.
\end{proof}
\begin{proposition}\label{log}
Assume that $u_0\in H^1,$ with $\frac{\omega_0}{r}\in L^2$ and $\omega_{0}\in\sqrt{\mathbb{L}}$. Let
$\rho_0\in L^2\cap L^\infty$ depending only on $(r,z)$ such that
$\hbox{\rm Supp }\rho_0$ does not intersect the axis $(Oz)$ and
$\Pi_z(\hbox{\rm Supp }\rho_0)$ is a compact set.
Then any smooth axisymmetric without swirl solution $(\rho,u)$ of  \eqref{bs}  satisfies
\begin{equation*}
\norm{\omega(t)}_{\sqrt{\mathbb{L}}}\leq Ce^{\exp^{Ct^{17}}}(\norm{\omega_0}_{\sqrt{\mathbb{L}}}+1).
\end{equation*}
Here the positive constant $C$ depends on the initial data.
\end{proposition}
\begin{proof}
Multiplying the vorticity equation \eqref{tourbillon} with $|\omega_\theta|^{p-2}\omega_\theta$ and performing integration in space, we get
\begin{equation*}
\begin{split}
&\frac1p\frac{\mathrm{d}}{\mathrm{d}t}\int_{\mathbb{R}^3} |\omega_\theta|^p\mathrm{d}x+(p-1)\int_{\mathbb{R}^3}|\nabla_h\omega_\theta|^{2}|\omega_\theta|^{p-2}\mathrm{d}x+\int_{\mathbb{R}^3}|\omega_\theta|^{p-2}\frac{\omega^{2}_\theta}{r^{2}}\mathrm{d}x\\
=&\int_{\mathbb{R}^3}\frac{u^r}{r} |\omega_\theta|^p\mathrm{d}x-\int_{\mathbb{R}^3}\partial_{r}\rho|\omega_\theta|^{p-2}\omega_\theta \mathrm{d}x.
\end{split}
\end{equation*}
For the first term in the last line of the above equality, we deuce by the H\"older inequality that
\begin{equation*}
\int_{\mathbb{R}^3}\frac{u_r}{r} |\omega_\theta|^p\mathrm{d}x\leq \Big\|\frac{u_{r}}{r}\Big\|_{L^{\infty}}\norm{\omega_\theta}^{p}_{L^p}.
\end{equation*}
For the second term, by the H\"older inequality,  we have
\begin{align*}
&-\int_{\mathbb{R}^3}\partial_{r}\rho|\omega_\theta|^{p-2}\omega_\theta \mathrm{d}x\\
\leq &\|\rho\|_{L^p}\big\||\omega_\theta|^{\frac {p-2}{2}}\big\|_{L^\frac{2p}{p-2}}\Big\||\omega_\theta|^{\frac {p-2}{2}}\frac{\omega_\theta}{r}\Big\|_{L^2}+(p-1)\|\rho\|_{L^p}\big\||\omega_\theta|^{\frac {p-2}{2}}\big\|_{L^\frac{2p}{p-2}}\Big\||\omega_\theta|^{\frac {p-2}{2}}\partial_r{\omega_\theta}\Big\|_{L^2}\nonumber\\ \leq&\frac{p-1}{2}\int_{\mathbb{R}^3}|\nabla_h\omega_\theta|^{2}|\omega_\theta|^{p-2}\mathrm{d}x+\frac12\int_{\mathbb{R}^3}|\omega_\theta|^{p-2}\frac{\omega^{2}_\theta}{r^{2}}\mathrm{d}x
+\frac{p}{2}\norm{\rho}^{2}_{L^p}
\norm{\omega_\theta}^{p-2}_{L^p}.
\end{align*}
Indeed, integrating by parts leads to
\begin{equation*}
\begin{split}
-\int_{\mathbb{R}^3}\partial_{r}\rho|\omega_\theta|^{p-2}\omega_\theta \mathrm{d}x=&-2\pi\int\partial_{r}\rho|\omega_\theta|^{p-2}\omega_\theta r\mathrm{d}r\mathrm{d}z\\
=&\int\rho|\omega_\theta|^{p-2}\omega_\theta \mathrm{d}r\mathrm{d}z+(p-1)\int\rho|\omega_\theta|^{p-2}\partial_{r}\omega_\theta r\mathrm{d}r\mathrm{d}z\\
=&\int_{\mathbb{R}^3}\rho|\omega_\theta|^{p-2}\frac{\omega_\theta}{r} \mathrm{d}x+(p-1)\int_{\mathbb{R}^3}\rho|\omega_\theta|^{p-2}\partial_{r}\omega_\theta \mathrm{d}x.
\end{split}
\end{equation*}
Therefore, by virtue of \mbox{Proposition \ref{energy}}, one has
\begin{equation*}
\frac{\mathrm{d}}{\mathrm{d}t}\norm{\omega_\theta(t)}^{2}_{L^{p}}\leq \Big\|\frac{u_{r}}{r}(t)\Big\|_{L^{\infty}}\norm{\omega_\theta(t)}^{2}_{L^p}+\frac{p}{2}\norm{\rho_0}^{2}_{L^{p}}.
\end{equation*}
The Gronwall inequality yields that
\begin{equation*}
\norm{\omega_\theta(t)}^{2}_{L^{p}}\leq e^{\int_0^t\|\frac{u_{r}}{r}(\tau)\|_{L^{\infty}}\intd\tau}\Big(\norm{\omega_{\theta}(0)}^{2}_{L^{p}}+\frac{p}{2}\norm{\rho_0}^{2}_{L^p}t\Big).
\end{equation*}
According to \eqref{bs-2} of \mbox{Proposition \ref{biot-s}} and to the fact that $\norm{\omega}_{L^{p}}=\norm{\omega_\theta}_{L^{p}}$, one can deduce that
\begin{equation*}
\norm{\omega(t)}^{2}_{\sqrt{\mathbb{L}}}\leq e^{C\exp^{Ct^{17}}}\Big(\norm{\omega_0}^{2}_{\sqrt{\mathbb{L}}}+\frac{t}{2}\norm{\rho_0}^{2}_{L^{2}\cap L^{\infty}}\Big).
\end{equation*}
This completes the proof.
\end{proof}
Proposition above  together with the well-known fact that $\|\nabla u\|_{L^2}\leq C\frac{p^2}{p-1}\|\omega\|_{L^p}$ for $p\in]1,\infty[$ yields that
$\sup_{p\geq2}\frac{\|\nabla u(t)\|_{L^p}}{p\sqrt{p}}$ is locally bounded with respect to time. But, the growth rate $p\sqrt{p}$ does not satisfies the dual Osgood modulus of continuity. This requires us to further refine the growth estimate of $\|\nabla u\|_{L^p}$ to get the growth estimate we require. Inspired by the Boot-Strap argument, we will establish the below proposition which is the heart in our proof based on the estimate of $\sup_{p\geq2}\frac{\|\nabla u(t)\|_{L^p}}{p\sqrt{p}}$.
\begin{proposition}\label{prop-fine}
Assume that $u_0\in H^1,$ with $\frac{\omega_0}{r}\in L^2$ and $\omega_{0}\in\sqrt{\mathbb{L}}$. Let
$\rho_0\in L^2\cap L^\infty$ depending only on $(r,z)$ such that
$\hbox{\rm Supp }\rho_0$ does not intersect the axis $(Oz)$ and
$\Pi_z(\hbox{\rm Supp }\rho_0)$ is a compact set.
Then any smooth axisymmetric without swirl solution $(\rho,u)$ of  \eqref{bs}  satisfies
\begin{equation}\label{eq-h-e}
\sup_{p\geq2}\int_0^t\sum_{q\geq0}2^{qs}\frac{\|u_{q}(\tau)\|_{L^p}}{p^{\frac32}}\intd\tau\leq C_1e^{\exp{C_1t^{17}}}\quad \text{for}\quad s\in[0,2[.
\end{equation}
In particular, we have
\begin{equation}\label{Loglip}
\sup_{2\leq p<\infty}\int_0^t\frac{\norm{\nabla u(\tau)}_{L^{p}}}{\sqrt{p}}\mathrm{d}\tau\leq C_2e^{\exp{C_2t^{17}}}.
\end{equation}
Here constants $C_1$ and $C_2$ depend on the initial data.
\end{proposition}
\begin{proof}
Applying the operator $\Delta^{h}_{q}$ to \eqref{bs} and using Duhamel formula we get
\begin{equation*}
\begin{split}
u_{q}(t,x)=&e^{t\Delta_{h}}u_{q}(0)-\int_0^te^{(t-\tau)\Delta_{h}}\Delta^{h}_q\mathcal{P}(u\cdot\nabla u)(\tau,x)\mathrm{d}\tau-\int_0^te^{(t-\tau)\Delta_{h}}\Delta^{h}_q\mathcal{P}\rho(\tau,x)e_z\mathrm{d}\tau,
\end{split}
\end{equation*}
where $u_{q}=\Delta^{h}_q u$ and $\mathcal{P}$ is the Leray projection on divergence free vector fields.

Notice that
\begin{equation}
u\cdot\nabla u=\omega\times u+\frac12\nabla|u|^2.
\end{equation}
Therefore,
$$\mathcal{P}(u\cdot\nabla u)=\mathcal{P}(\omega\times u).
$$
According to Proposition \ref{heat}, we have the following estimate for $q\geq0$
$$
\|e^{t\Delta_{h}}\Delta^{h}_q f\|_{L^p(\RR^3)}\leq\norm{\|e^{t\Delta_{h}}\Delta^{h}_q f\|_{L^p(\RR_{h}^2)}}_{L^{p}(\RR_{v})}\le Ce^{-ct2^{2q}}\|\Delta^{h}_q  f\|_{L^p(\RR^3)}.
$$
Therefore,  we have that for $q\geq 0$
\begin{equation*}
\|u_{q}\|_{L^1_tL^p}\leq C 2^{-2q}\|u_{q}(0)\|_{L^p}+Cp2^{-2q}\int_0^t\|\Delta^{h}_q (\omega\times u)(\tau)\|_{L^p}\mathrm{d}\tau+Cp2^{-2q}\|\Delta^{h}_{q}\rho\|_{L^1_t L^p}.
\end{equation*}
Multiplying the above inequality by $2^{qs}$ and summing over $q\geq0$, we readily obtain that
\begin{align}\label{Log-11}
&\sum_{q\geq0}2^{qs}\|u_{q}\|_{L^1_tL^p}\nonumber\\
\lesssim & \sum_{q\geq0}2^{q(s-2)}\|u_{q}(0)\|_{L^p}+p\int_0^t\sum_{q\geq0}2^{q(s-2)}\|\Delta^{h}_q (\omega\times u)(\tau)\|_{L^p}\mathrm{d}\tau+p\sum_{q\geq0}2^{q(s-2)}\|\Delta^{h}_{q}\rho\|_{L^1_t L^p}\nonumber\\
:=&I_1+I_2+I_3.
\end{align}
For the first term $I_1$, by virtue of the imbedding theorem, we have that for $s\in]1,2[$
\begin{align}\label{eq.I1}
I_1\leq C\|u_0\|_{L^p}\leq& C\big(\|u_0\|_{L^2}+\|u_0\|_{L^\infty}\big)\nonumber\\
\leq& C\big(\|u_0\|_{L^2}+\|\omega_0\|_{L^4}\big).
\end{align}
For the third term $I_3$, by Proposition \ref{energy}, we obtain that for $s\in]1,2[$,
\begin{equation}\label{eq.I3}
I_3\leq Cpt\|\rho\|_{L^\infty_tL^p}\leq Cpt\|\rho_0\|_{L^2\cap L^\infty}.
\end{equation}
Now we tackle the second parentheses of $I_{2}$. the H\"older inequality and the interpolation inequality allow us to conclude that for $s\in]1,2[$
\begin{align*}
\sum_{q\geq0}2^{q(s-2)}\|\Delta^{h}_q (\omega\times u)\|_{L^p}\leq&C\|\omega\times u\|_{L^p}\\
\leq&C\|\omega\|_{L^p}\|u\|_{L^\infty}\\
\leq&C\|\omega\|_{L^p}\big(\|u\|_{L^2}+\|\omega(t)\|_{L^4}\big)
\end{align*}
Whence,
\begin{equation}\label{eq.I2}
I_{2}\leq Cp\big(\|u_0\|_{L^2}+\|\omega\|_{L^\infty_tL^4}\big)\int_0^t\norm{\omega(\tau)}_{L^{p}}\intd\tau.
\end{equation}
Putting together \eqref{eq.I1}, \eqref{eq.I3} and \eqref{eq.I2}, we finally get for $s\in]1,2[$
\begin{align*}
\sum_{q\geq0}2^{qs}\|u_{q}\|_{L^1_tL^p}
\leq  C\big(\|u_0\|_{L^2}+\|\omega_0\|_{L^4}\big)+Cpt\norm{\rho_{0}}_{L^{2}\cap L^{\infty}}
+Cp\big(\|u_0\|_{L^2}+\|\omega\|_{L^\infty_tL^4}\big)\int_0^t\norm{\omega(\tau)}_{L^{p}}\intd\tau.
\end{align*}
It follows from \mbox{Proposition \ref{energy}}, \mbox{Proposition \ref{prop-cruc}} and \mbox{Proposition \ref{log}} that
\begin{align}\label{eq-add-p}
&\sup_{p\geq2}\int_0^t\sum_{q\geq0}2^{qs}\frac{\|u_{q}(\tau)\|_{L^p}}{p^{\frac32}}\intd\tau\nonumber\\
\leq & C\big(\|u_0\|_{L^2}+\|\omega_0\|_{L^4}\big)+Ct\|\rho_0\|_{L^2\cap L^\infty}+C\big(\|u_0\|_{L^2}+\|\omega\|_{L^\infty_tL^4}\big)\int_0^t\norm{\omega(\tau)}_{\sqrt{\mathbb{L}}}\intd\tau\nonumber\\
\leq&Ce^{\exp{Ct^{17}}}.
\end{align}
On the other hand, by the sharp interpolation inequality, we have
\begin{align*}
\sum_{q\geq0}2^{q}\|\Delta_q^hu\|_{L^p}
\leq&C\big\|\Lambda^{\frac{3}{4}}_hu\big\|_{L^p}^{\frac{2}{3}}\Big(\sum_{q\geq0}2^{\frac{3}{2}q}\|\Delta_q^hu\|_{L^p}\Big)^{\frac{1}{3}}\nonumber\\
\leq&C\Big(\big\|\Lambda^{\frac{3}{4}}_hu\big\|_{L^2}+\|\nabla u\|_{L^{15}}\Big)^{\frac{2}{3}}\Big(\sum_{q\geq0}2^{\frac{3}{2}q}\|\Delta_q^hu\|_{L^p}\Big)^{\frac{1}{3}}\nonumber\\
\leq&C\Big(\|u\|_{H^1}+\|\omega\|_{L^{15}}\Big)^{\frac{2}{3}}\Big(\sum_{q\geq0}2^{\frac{3}{2}q}{\|\Delta_q^hu\|_{L^p}}\Big)^{\frac{1}{3}}.
\end{align*}
Consequently, by using \eqref{eq-add-p}, \mbox{Proposition \ref{energy}},  \mbox{Proposition \ref{prop-cruc}} and \mbox{Proposition \ref{log}}, we get that
\begin{align}\label{eq-add-pp}
&\int_0^t\frac{\|\nabla_h u(\tau)\|_{L^p}}{\sqrt{p}}\intd\tau\nonumber\\
\leq&C\int_0^t\sum_{q\geq0}2^{q}\frac{\|\Delta_q^h u(\tau)\|_{L^p}}{\sqrt{p}}\intd\tau
+C\int_0^t\|u(\tau)\|_{L^p}\intd\tau\nonumber\\
\leq & Ct^{\frac{2}{3}}\Big(\|u\|_{L^\infty_tH^1}+\|\omega\|_{L^\infty_tL^{15}}\Big)^{\frac{2}{3}}\Big(\sum_{q\geq0}\int_0^t2^{\frac{3}{2}q}\frac{\|\Delta_q^hu(\tau)\|_{L^p}}{p^{\frac32}}\intd\tau\Big)^{\frac{1}{3}} +Ct\big(\|u\|_{L^\infty_tL^2}+\|\omega\|_{L^\infty_tL^4}\big)\nonumber\\
\leq& Ce^{\exp{Ct^{17}}}.
\end{align}
Note that
\begin{equation*}
\partial_zu=\partial_z(u_re_r+u_z e_z)=\partial_zu_re_r+\partial_zu_ze_{z}=-\omega_\theta e_r+\partial_ru_ze_r-(\partial_1u_1+\partial_2u_2)e_{z}.
\end{equation*}
Thus, by using \eqref{eq-add-pp} and \mbox{Proposition \ref{log}}, we end up with
\begin{align*}
\sup_{p\geq2}\int_0^t\frac{\norm{\nabla u(\tau)}_{L^p}}{\sqrt{p}}\mathrm{d}\tau
\leq &C\sup_{p\geq2}\int_0^t\frac{\norm{\nabla_{h} u(\tau)}_{L^p}}{\sqrt{p}}\intd\tau+C\int_0^t\sup_{p\geq2}\frac{\norm{\omega(\tau)}_{L^p}}{\sqrt{p}}\mathrm{d}\tau \nonumber\\ \leq&Ce^{\exp{Ct^{17}}}.
\end{align*}
This implies the desired result \eqref{Loglip}.
\end{proof}

\begin{proposition}\label{prop-pri}
Let $(\rho,u)$ be a smooth solution of the system \eqref{bs}.  Assume the initial data $(\rho_0,u_0 )$ satisfies the conditions stated in \mbox{Theorem \ref{global}}.
Then there exists a constant $C$ such that for all $t\in[0,T]$
\begin{equation}\label{high-es}
\norm{\nabla\rho(t)}^2_{L^2}+\|\partial_z\omega_\theta(t)\|^2_{L^2}+\int^{t}_{0}\norm{\nabla_h\partial_z\omega_\theta(\tau)}^{2}_{L^2}\mathrm{d}\tau+\int_0^t\Big\|\frac{\partial_z\omega_\theta(\tau)}{r}\Big\|^{2}_{L^2}\intd\tau
\leq C_3e^{e^{e^{\exp{C_3t^{17}}}}}.
\end{equation}
Moreover, we have
\begin{equation}\label{high-es-2}
\norm{\nabla u}_{L^1_tL^\infty}\leq C_4e^{e^{e^{\exp{C_4t^{17}}}}}.
\end{equation}
Here the positive constants $C_3$ and $C_4$ depend on the initial data $(u_0,\rho_0)$
\end{proposition}
\begin{proof}
Applying the operator $\partial_z$ to \eqref{tourbillon} yields
\begin{equation}\label{partial-z-tourbillon}
\bigl(\partial_t +u\cdot\nabla\bigr)\partial_z\omega_\theta-\Delta_{h}\partial_z\omega_\theta
+\frac{\partial_z\omega_\theta}{r^2} =-\partial_z\partial_{r}\rho+\partial_z\Big(\frac{u_r}{r}\Big)\omega_\theta+\frac{u_r}{r}\partial_z\omega_\theta-\partial_zu_r\partial_r\omega_\theta
-\partial_zu_z\partial_z\omega_\theta.
\end{equation}
By the standard energy method, one can conclude that
\begin{align*}
&\frac12\dtd\|\partial_z\omega_\theta(t)\|^{2}_{L^2}+\|\nabla_h\partial_z\omega_\theta(t)\|^{2}_{L^2}+\Big\|\frac{\partial_z\omega_\theta(t)}{r}\Big\|^{2}_{L^2}\nonumber\\
=&-\int_{\mathbb{R}^3}\partial_z\partial_{r}\rho\partial_z\omega_\theta\intd x+\int_{\mathbb{R}^3}\partial_z\Big(\frac{u_r}{r}\Big)\omega_\theta\partial_z\omega_\theta\intd x+\int_{\mathbb{R}^3}\frac{u_r}{r}\partial_z\omega_\theta\partial_z\omega_\theta\intd x\nonumber\\
&-\int_{\mathbb{R}^3}\partial_zu_r\partial_r\omega_\theta\partial_z\omega_\theta\intd x-\int_{\mathbb{R}^3}\partial_zu_z\partial_z\omega_\theta\partial_z\omega_\theta\intd x\nonumber\\
:=&I_1+I_2+I_3+I_4+I_5.
\end{align*}
Since the support of $\rho$ ensures
\begin{align*}
-\int_{\mathbb{R}^3}\partial_z\partial_{r}\rho\partial_z\omega_\theta\intd x=&-2\pi\int_{0}^\infty\int_{\mathbb{R}}\partial_z\partial_{r}\rho\partial_z\omega_\theta r\intd r\intd z\nonumber\\
=&2\pi\int_{0}^\infty\int_{\mathbb{R}}\partial_z\rho\partial_z\omega_\theta \intd r\intd z+2\pi\int_{0}^\infty\int_{\mathbb{R}}\partial_z\rho\partial_r\partial_z\omega_\theta r\intd r\intd z\nonumber\\
=&\int_{\mathbb{R}^{3}}\partial_z\rho\frac{\partial_z\omega_\theta}{r} \intd x+\int_{\mathbb{R}^{3}}\partial_z\rho{\partial_r\partial_z\omega_\theta} \intd x,
\end{align*}
then the first term $I_1$ may be bounded by
\begin{align*}
&\|\partial_z\rho\|_{L^2}\Big\|\frac{\partial_z\omega_\theta}{r}\Big\|_{L^2}+\|\partial_z\rho\|_{L^2}\|\partial_r\partial_z\omega_\theta\|_{L^2}\nonumber\\
\leq&2\|\partial_z\rho\|^{2}_{L^2}+\frac14\Big(\Big\|\frac{\partial_z\omega_\theta}{r}\Big\|^{2}_{L^2}+\|\partial_r\partial_z\omega_\theta\|^{2}_{L^2}\Big).
\end{align*}
We observe that
\begin{align*}
\int_{\mathbb{R}^3}\partial_z\Big(\frac{u_r}{r}\Big)\omega_\theta\partial_z\omega_\theta\intd x=&\frac12\int_{\mathbb{R}^3}\partial_z\Big(\frac{u_r}{r}\Big)\partial_z(\omega_\theta)^2\intd x\nonumber\\
=&-\frac12\int_{\mathbb{R}^3}\partial^{2}_z\Big(\frac{u_r}{r}\Big)(\omega_\theta)^2\intd x.
\end{align*}
Therefore, by the H\"older inequality, the Young inequality and Lemma \ref{prop-identity}, we obtain
\begin{align*}
I_2\leq&\Big\|\partial^{2}_z\Big(\frac{u_r}{r}\Big)\Big\|_{L^2}\|\omega_\theta\|_{L^4}^2\nonumber\\
\leq& C\Big\|\partial_z\Big(\frac{\omega_\theta}{r}\Big)\Big\|_{L^2}\|\omega_\theta\|_{\sqrt{\mathbb{L}}}^2\nonumber\\
\leq& C\|\omega_\theta\|_{\sqrt{\mathbb{L}}}^4+\frac18\Big\|\partial_z\Big(\frac{\omega_\theta}{r}\Big)\Big\|^2_{L^2}.
\end{align*}
By the H\"older inequality and Proposition \ref{biot-s}, we have
\begin{equation*}
I_3\leq\Big \|\frac{u_r}{r}\Big\|_{L^\infty}\|\partial_z\omega_\theta\|^{2}_{L^2}\leq C \Big\|\frac{\omega_\theta}{r}\Big\|^{\frac{1}{2}}_{L^2}\Big\|\nabla_h\Big(\frac{\omega_\theta}{r}\Big)\Big\|^{\frac{1}{2}}_{L^2}\|\partial_z\omega_\theta\|^{2}_{L^2}.
\end{equation*}
As for the term $I_4$, we observe that
\begin{align*}
-\int_{\mathbb{R}^3}\partial_zu_r\partial_r\omega_\theta\partial_z\omega_\theta\intd x=&\int_{\mathbb{R}^3}\omega_\theta\partial_r\omega_\theta\partial_z\omega_\theta\intd x-\int_{\mathbb{R}^3}\partial_ru_z\partial_r\omega_\theta\partial_z\omega_\theta\intd x\nonumber\\
=&-\frac12\int_{\mathbb{R}^3}\omega_\theta^2\Big(\partial_r\partial_z\omega_\theta+\frac{\partial_z\omega_\theta}{r}\Big)\intd x-\int_{\mathbb{R}^3}\partial_ru_z\partial_r\omega_\theta\partial_z\omega_\theta\intd x\nonumber\\
:=&I_{4}^1+I^2_4.
\end{align*}
By the H\"older inequality and the Young inequality, we get
\begin{align*}
I_{4}^1\leq&\frac12\|\omega_\theta\|_{L^4}^2\Big\|\Big(\partial_r\partial_z\omega_\theta+\frac{\partial_z\omega_\theta}{r}\Big)\Big\|_{L^2}\nonumber\\
\leq&C\|\omega_\theta\|_{L^4}^4+\frac{1}{16}\Big(\|\nabla_h\partial_z\omega_\theta\|^{2}_{L^2}+\Big\|\frac{\partial_z\omega_\theta}{r}\Big\|^{2}_{L^2}\Big)\nonumber\\
\leq&C\|\omega_\theta\|_{\sqrt{\mathbb{L}}}^4+\frac{1}{16}\Big(\|\nabla_h\partial_z\omega_\theta\|^{2}_{L^2}+\Big\|\frac{\partial_z\omega_\theta}{r}\Big\|^{2}_{L^2}\Big).
\end{align*}
In the light of  Lemma \ref{lema.2}, we have
\begin{align*}
I_4^{2}\leq & C\|\partial_ru_z\|^{\frac{1}{2}}_{L^2}\|\nabla_h\partial_ru_z\|^{\frac{1}{2}}_{L^2}\|\partial_r\omega_\theta\|^{\frac{1}{2}}_{L^2}\|\partial_r\partial_z\omega_\theta\|^{\frac{1}{2}}_{L^2}
\|\partial_z\omega_\theta\|^{\frac{1}{2}}_{L^2}\|\nabla_h\partial_z\omega_\theta\|^{\frac{1}{2}}_{L^2}\nonumber\\
\leq&C\|\omega\|^{\frac{1}{2}}_{L^2}\|\nabla_h\omega\|_{L^2}\|\nabla_h\partial_z\omega_\theta\|_{L^2}
\|\partial_z\omega_\theta\|^{\frac{1}{2}}_{L^2}\nonumber\\
\leq&C\|\omega\|_{L^2}\|\nabla_h\omega\|^{2}_{L^2}\|\partial_z\omega_\theta\|_{L^2}+\frac18\|\nabla_h\partial_z\omega_\theta\|^{2}_{L^2}
\nonumber\\
\leq&C\|\omega\|^2_{L^2}\|\nabla_h\omega\|^{2}_{L^2}+C\|\nabla_h\omega\|^{2}_{L^2}\|\partial_z\omega_\theta\|^2_{L^2}+\frac18\|\nabla_h\partial_z\omega_\theta\|^{2}_{L^2}.
\end{align*}
For the last term $I_5$, the incompressible condition that $\partial_ru_r+\frac{u_r}{r}+\partial_zu_z=0$ guarantees that
\begin{equation*}
I_5=\int_{\mathbb{R}^3}\partial_ru_r\partial_z\omega_\theta\partial_z\omega_\theta\intd x+\int_{\mathbb{R}^3}\frac{u_r}{r}\partial_z\omega_\theta\partial_z\omega_\theta\intd x.
\end{equation*}
Integrating by parts gives
\begin{align*}
\int_{\mathbb{R}^3}\partial_ru_r\partial_z\omega_\theta\partial_z\omega_\theta\intd x=&2\pi\int_0^\infty\int_{\RR}\partial_ru_r\partial_z\omega_\theta\partial_z\omega_\theta r\intd r\intd z\nonumber\\
=&-2\pi\int_0^\infty\int_{\RR}u_r\partial_z\omega_\theta\partial_z\omega_\theta \intd r\intd z-4\pi\int_0^\infty\int_{\RR}u_r\partial_r\partial_z\omega_\theta\partial_z\omega_\theta r\intd r\intd z\nonumber\\
=&-\int_{\mathbb{R}^3}\frac{u_r}{r}\partial_z\omega_\theta\partial_z\omega_\theta\intd x-2\int_{\RR^3}u_r\partial_r\partial_z\omega_\theta\partial_z\omega_\theta\intd x.
\end{align*}
Consequently, the H\"older inequality and the Young inequality enable us to conclude that
\begin{align*}
I_5=&-2\int_{\RR^3}u_r\partial_r\partial_z\omega_\theta\partial_z\omega_\theta\intd x\nonumber\\
\leq &2\|u\|_{L^\infty}\|\partial_z\omega_\theta\|_{L^2}\|\nabla_h\partial_z\omega_\theta\|_{L^2}\nonumber\\
\leq&C\|\omega\|_{\sqrt{\mathbb{L}}}^2\|\partial_z\omega_\theta\|^{2}_{L^2}+\frac18\|\nabla_h\partial_z\omega_\theta\|^{2}_{L^2}.
\end{align*}
Here we have used the fact that
\begin{align*}
\|u\|_{L^\infty}\leq\sum_{q\geq-1}\|\Delta_qu\|_{L^\infty}\leq&C\sum_{q\geq-1}2^{\frac{1}{2}q}\|\Delta_qu\|_{L^6}\nonumber\\
\leq&C\|u\|_{L^6}+C\sum_{q\geq0}2^{-\frac{1}{2}q}\|\Delta_q\nabla u\|_{L^6}\nonumber\\
\leq&C\big(\|\omega\|_{L^2}+\|\omega\|_{L^6}\big)\leq C\|\omega\|_{\sqrt{\mathbb{L}}}.
\end{align*}
Gathering these estimates, we finally obtain that
\begin{align}\label{eq.v-e}
&\dtd\|\partial_z\omega_\theta(t)\|^{2}_{L^2}+\|\nabla_h\partial_z\omega_\theta(t)\|^{2}_{L^2}
+\Big\|\frac{\partial_z\omega_\theta(t)}{r}\Big\|^{2}_{L^2}\nonumber\\
\leq&C\Big(\Big\|\frac{\omega_\theta}{r}\Big\|^{\frac{1}{2}}_{L^2}
\Big\|\nabla_h\Big(\frac{\omega_\theta}{r}\Big)\Big\|^{\frac{1}{2}}_{L^2}
+\|\nabla_h\omega\|^{2}_{L^2}+\|\omega\|_{\sqrt{\mathbb{L}}}^2\Big)\|\partial_z\omega_\theta\|^{2}_{L^2}\nonumber\\
&+C\Big(\|\omega\|^2_{L^2}\|\nabla_h\omega\|^{2}_{L^2}+\|\omega\|^4_{\sqrt{\mathbb{L}}}\Big)+C\|\nabla\rho\|_{L^2}^2.
\end{align}
Next, applying the differential operator $\partial_i$ to the density equation with $i=1,2,3$, one has
\begin{equation*}
\bigl(\partial_t+u\cdot\nabla\bigr)\partial_i\rho=-\partial_iu_r\partial_r\rho-\partial_iu_z\partial_z\rho_i.
\end{equation*}
Taking $L^2$-inner product of the above equation with $\partial_i\rho$, we immediately obtain that
\begin{align*}
\frac12\dtd \|\partial_i\rho(t)\|_{L^2}^2=&-\int_{\mathbb{R}^3}\partial_iu_r\partial_r\rho\partial_i\rho\intd x-\int_{\mathbb{R}^3}\partial_iu_z\partial_z\rho\partial_i\rho\intd x\nonumber\\
\leq&2\|\nabla u\|_{L^\infty}\|\nabla\rho\|^{2}_{L^2}.
\end{align*}
This implies that
\begin{equation*}
\dtd \|\nabla\rho(t)\|_{L^2}^2\leq C\|\nabla u\|_{L^\infty}\|\nabla\rho\|^{2}_{L^2}.
\end{equation*}
The Gronwall inequality provides
\begin{equation}\label{eq.exprho}
\|\nabla\rho(t)\|_{L^2}^2\leq \|\nabla\rho_0\|^{2}_{L^2}e^{C\int_0^t\|\nabla u(\tau)\|_{L^\infty}\intd\tau}.
\end{equation}
On the other hand, by taking advantage of Lemma \ref{log}, we know that for all $\epsilon\in]0,\frac{1}{16}[,$
\begin{align*}
\int^{t}_{0}\norm{\nabla u(\tau)}_{L^\infty(\mathbb{R}^3)}\mathrm{d}t\leq& C+C\sup_{2\leq j<\infty}\int^{t}_{0}\frac{\norm{S_{j}\nabla u(\tau)}_{L^{\infty}(\mathbb{R}^3)}}{\sqrt{j}}\mathrm{d}t
\Big(\log\big(e+\norm{\nabla u}_{L^{1}_{t}(B_{\infty,\infty}^{\epsilon}(\mathbb{R}^3))}\big)\Big)^{\frac12}\\
\leq& C+C\Big(\sup_{2\leq p<\infty}\int^{t}_{0}\frac{\norm{\nabla u(\tau)}_{L^{p}(\mathbb{R}^3)}}{\sqrt{p}}\mathrm{d}t\Big)^2
+\log\big(e+\norm{\nabla u}_{L^{1}_{t}(B_{\infty,\infty}^{\epsilon}(\mathbb{R}^3))}\big).
\end{align*}
Inserting this into \eqref{eq.exprho}, we get
\begin{align}\label{eq.nabla-rho}
\|\nabla\rho(t)\|_{L^2}^2\leq& \|\nabla\rho_0\|^{2}_{L^2}e^{C}e^{Ch(t)}\big(e+\norm{\nabla u}_{L^{1}_{t}(B_{\infty,\infty}^{\epsilon}(\mathbb{R}^3))}\big)\nonumber\\
\leq&C \|\nabla\rho_0\|^{2}_{L^2}e^{Ch(t)}+C \|\nabla\rho_0\|^{2}_{L^2}e^{Ch(t)}\norm{\nabla u}_{L^{1}_{t}(B_{\infty,\infty}^{\epsilon})}\nonumber\\
\leq&Ce^{Ch(t)}+C e^{Ch(t)}\norm{\nabla u}_{L^{1}_{t}(B_{\infty,\infty}^{\epsilon})},
\end{align}
where $\sqrt{h(t)}:=\sup_{2\leq p<\infty}\int^{t}_{0}\frac{\norm{\nabla u(\tau)}_{L^{p}}}{\sqrt{p}}\mathrm{d}\tau.$

Putting this estimate together with \eqref{eq.v-e} yields
\begin{align}\label{eq.7.2}
&\dtd\|\partial_z\omega_\theta(t)\|^{2}_{L^2}+\|\nabla_h\partial_z\omega_\theta(t)\|^{2}_{L^2}
+\Big\|\frac{\partial_z\omega_\theta(t)}{r}\Big\|^{2}_{L^2}\nonumber\\
\leq&C\Big(\Big\|\frac{\omega_\theta}{r}\Big\|^{\frac{1}{2}}_{L^2}
\Big\|\nabla_h\Big(\frac{\omega_\theta}{r}\Big)\Big\|^{\frac{1}{2}}_{L^2}
+\|\nabla_h\omega\|^{2}_{L^2}+\|\omega\|_{\sqrt{\mathbb{L}}}^2\Big)\|\partial_z\omega_\theta\|^{2}_{L^2}\nonumber\\
&+C\Big(\|\omega\|^2_{L^2}\|\nabla_h\omega\|^{2}_{L^2}+\|\omega\|^4_{\sqrt{\mathbb{L}}}+Ce^{Ch(t)}\Big)+C e^{Ch(t)}\norm{\nabla u}_{L^{1}_{t}(B_{\infty,\infty}^{\epsilon})}.
\end{align}
Employing the H\"older inequality,   \mbox{Proposition \ref{prop-cruc}} and \mbox{Proposition \ref{log}}, we get that
\begin{align*}
&\int_0^t\Big(\Big\|\frac{\omega_\theta}{r}(\tau)\Big\|^{\frac{1}{2}}_{L^2}
\Big\|\nabla_h\Big(\frac{\omega_\theta}{r}\Big)(\tau)\Big\|^{\frac{1}{2}}_{L^2}
+\|\nabla_h\omega(\tau)\|^{2}_{L^2}+\|\omega(\tau)\|_{\sqrt{\mathbb{L}}}^2\Big)\intd\tau\nonumber\\
\leq&\sqrt{t}\Big\|\frac{\omega_\theta}{r}\Big\|^{\frac{1}{2}}_{L^\infty_tL^2}
\Big\|\nabla_h\Big(\frac{\omega_\theta}{r}\Big)\Big\|^{\frac{1}{2}}_{L^2_tL^2}
+\|\nabla_h\omega\|^{2}_{L^2_tL^2}+t\|\omega\|_{L^\infty_t\sqrt{\mathbb{L}}}^2\nonumber\\
\leq&Ce^{\exp{Ct^{17}}}.
\end{align*}
By using the Gronwall inequality,   \mbox{Proposition \ref{prop-cruc}}, \mbox{Proposition \ref{log}} and \mbox{Proposition \ref{prop-fine}}, we readily obtain
\begin{align}\label{eq.7.3}
&\|\partial_z\omega_\theta(t)\|^{2}_{L^2}+\int_0^t\|\nabla_h\partial_z\omega_\theta(\tau)\|^{2}_{L^2}\intd\tau
+\int_0^t\Big\|\frac{\partial_z\omega_\theta(\tau)}{r}\Big\|^{2}_{L^2}\intd\tau\nonumber\\
\leq& Ce^{e^{\exp{Ct^{17}}}}\left(\|\partial_z\omega_\theta(0)\|^{2}_{L^2}+\|\omega\|^2_{L^\infty_tL^2}\|\nabla_h\omega\|^{2}_{L^2_tL^2}
+t\|\omega\|^4_{L^\infty_t\sqrt{\mathbb{L}}}+Cte^{Ch(t)}+Cte^{Ch(t)}\norm{\nabla u}_{L^{1}_{t}(B_{\infty,\infty}^{\epsilon})}\right)\nonumber\\
\leq& Ce^{e^{\exp{Ct^{17}}}}+Ce^{e^{\exp{Ct^{17}}}}\norm{\nabla u}_{L^{1}_{t}(B_{\infty,\infty}^{\epsilon})}.
\end{align}
This together with \eqref{eq.nabla-rho} and \mbox{Proposition \ref{log}} yields
\begin{align}\label{eq.m-n}
&\|\partial_z\omega_\theta(t)\|^{2}_{L^2}+\|\nabla\rho(t)\|_{L^2}^2+\int_0^t\|\nabla_h\partial_z\omega_\theta(\tau)\|^{2}_{L^2}\intd\tau
+\int_0^t\Big\|\frac{\partial_z\omega_\theta(\tau)}{r}\Big\|^{2}_{L^2}\intd\tau\nonumber\\
\leq& Ce^{e^{\exp{Ct^{17}}}}+Ce^{e^{\exp{Ct^{17}}}}\norm{\nabla u}_{L^{1}_{t}(B_{\infty,\infty}^{\epsilon})}.
\end{align}
Now, we tackle with the integral term $\int_0^t\norm{\nabla u(\tau)}_{B_{\infty,\infty}^{\epsilon}}\intd\tau$. It is clear that
\begin{align}\label{eq-split}
\norm{\nabla u(t)}_{B_{\infty,\infty}^{\epsilon}}\leq \norm{\partial_zu(t)}_{B_{\infty,\infty}^{\epsilon}}+\norm{\nabla_h u(t)}_{B_{\infty,\infty}^{\epsilon}}.
\end{align}
On the one hand, with the help of Lemma \ref{lem-anstro} and the fact that $\omega=\omega_\theta e_\theta$, one can infer that for $\epsilon\in]0,\frac{1}{16}[,$
\begin{align*}
&\norm{\partial_zu(\tau)}_{B_{\infty,\infty}^{\epsilon}}\\
\leq&C\|\partial_zu\|_{L^2}+C\|\Lambda^\epsilon \partial_zu\|_{L^\infty}\nonumber\\
\leq&C\|\partial_zu\|_{L^2}+C\|\Lambda^\epsilon \partial_zu\|_{L^2}+C\|\Lambda^\epsilon \Lambda_h^{1+\epsilon}\partial_zu\|_{L^2}+C\|\Lambda_v^{1-\epsilon}\Lambda^{\epsilon} u\|_{L^2}+C\|\Lambda_v^{1-2\epsilon}\Lambda_h^{1+\epsilon}\Lambda^\epsilon \partial_zu\|_{L^2}\nonumber\\
\leq&C\|u\|_{H^1}+C\|\partial_z\omega\|_{L^2}+C\|\nabla_h \partial_z\omega\|_{L^2}\nonumber\\
\leq&C\|u\|_{H^1}+C\|\partial_z\omega_\theta\|_{L^2}+C\|\nabla_h \partial_z\omega_\theta\|_{L^2}+C\Big\| \partial_z\Big(\frac{\omega_\theta}{r}\Big)\Big\|_{L^2}.
\end{align*}
Therefore, we immediately obtain
\begin{align}\label{eq.7.4}
\norm{\partial_zu(\tau)}_{L^1_tB_{\infty,\infty}^{\epsilon}}\leq Ct\|u\|_{L^\infty_tH^1}+C\sqrt{t}\|\partial_z\omega\|_{L^2_tL^2}+C\sqrt{t}C\|\nabla_h \partial_z\omega_\theta\|_{L^2_tL^2}+C\sqrt{t}\Big\| \partial_z\Big(\frac{\omega_\theta}{r}\Big)\Big\|_{L^2_tL^2}.
\end{align}
Next, we need to introduce an useful lemma in order to tackle with another part $\norm{\nabla_h u(t)}_{B_{\infty,\infty}^{\epsilon}}$.
\begin{lemma}[\cite{MZ2012}]\label{smoothing}
Let $s_{1},s_{2}\in \RR$ and $p\in[2,\infty[$. Assume that $(\rho,u )$ be a smooth  solution  of the system \eqref{bs}, then there holds that
\begin{equation*}
\norm{u}_{L^{1}_{t}B^{s_{1}+2,s_{2}}_{p,1}}
\lesssim  \norm{u_{0}}_{B_{p,1}^{s_{1},s_{2}}}+\norm{u}_{L^{1}_{t}B_{p,1}^{s_{1},s_{2}}}+\norm{u\otimes u}_{L^{1}_{t}B_{p,1}^{s_{1}+1,s_{2}}\cap L^{1}_{t}B_{p,1}^{s_{1},s_{2}+1}}+\norm{\rho}_{L^{1}_{t}B_{p,1}^{s_{1},s_{2}}}.
\end{equation*}
\end{lemma}
With the help of Lemma \ref{smoothing} and the Bernstein inequality, the term $\norm{\nabla_h u(\tau)}_{B_{\infty,\infty}^{\epsilon}}$ can be bounded  as follows:
\begin{align}\label{eq.7.6}
\norm{\nabla_h u}_{L^1_tB_{\infty,\infty}^{\epsilon}}\leq&C\norm{u}_{L^{1}_{t}B^{1+\frac{2}{p}+\epsilon,\frac{1}{p}+\epsilon}_{p,1}}\nonumber\\
\leq&C\norm{u_{0}}_{B_{p,1}^{-1+\frac2p+\epsilon,\frac1p+\epsilon}}+\norm{u}_{L^{1}_{t}B_{p,1}^{-1+\frac2p+\epsilon,\frac1p+\epsilon}}+\norm{u\otimes u}_{L^{1}_{t}B_{p,1}^{\frac2p+\epsilon,\frac1p+\epsilon}\cap L^{1}_{t}B_{p,1}^{-1+\frac2p+\epsilon,1+\frac1p+\epsilon}}\nonumber\\
&+\norm{\rho}_{L^{1}_{t}B_{p,1}^{-1+\frac2p+\epsilon,\frac1p+\epsilon}}\nonumber\\
\leq&C\norm{u_{0}}_{B_{2,1}^{\epsilon,\frac12+\epsilon}}+\norm{u}_{L^{1}_{t}B_{2,1}^{\epsilon,\frac12+\epsilon}}+\norm{u\otimes u}_{L^{1}_{t}B_{2,1}^{1+\epsilon,\frac12+\epsilon}\cap L^{1}_{t}B_{2,1}^{\epsilon,\frac32+\epsilon}}\nonumber\\
&+\norm{\rho}_{L^{1}_{t}B_{2,1}^{\epsilon,\frac12+\epsilon}}\nonumber\\
\leq&C\norm{u_{0}}_{H^1}+\norm{u}_{L^{1}_{t}H^1}+\norm{u\otimes u}_{L^{1}_{t}B_{2,1}^{1+\epsilon,\frac12+\epsilon}\cap L^{1}_{t}B_{2,1}^{\epsilon,\frac32+\epsilon}}+\norm{\rho}_{L^{1}_{t}H^1}.
\end{align}
According to the Banach algebra property of Lemma \ref{properties}, one has that for $\epsilon\in]0,\frac{1}{16}[,$
\begin{align*}
\norm{u\otimes u}_{L^{1}_{t}B_{2,1}^{1+\epsilon,\frac12+\epsilon}}\leq C\norm{u}^{2}_{L^{2}_{t}H^{1+2\epsilon,\frac12+2\epsilon}}
\leq C\|u\|^{2}_{L^2_tH^1}+\|\nabla_h\omega\|^2_{L^2_tL^2}.
\end{align*}
On the other hand, Lemma \ref{properties} allows us to conclude that
\begin{align*}
\norm{u\otimes u}_{ L^{1}_{t}B_{2,1}^{\epsilon,\frac32+\epsilon}}\leq &C\norm{u\otimes u}_{ L^{1}_{t}H^{1+\epsilon,\frac32+2\epsilon}}\nonumber\\
\leq&C\norm{u}^2_{ L^{2}_{t}H^{1+\epsilon,\frac32+2\epsilon}}.
\end{align*}
Since the interpolation theorem and the fact that $\omega=\omega_\theta e_\theta$ guarantee that there exist $\beta\in]0,1[$ such that for $\epsilon\in]0,\frac{1}{16}[,$
\begin{align*}
\norm{u}_{H^{1+\epsilon,\frac32+2\epsilon}}\leq &C\|u\|_{L^2}+C\|\Lambda_h^{1+\epsilon}u\|_{L^2}+C\|\Lambda_v^{\frac32+2\epsilon}u\|_{L^2}+C\|\Lambda_h^{1+\epsilon}\Lambda_v^{\frac32+2\epsilon}u\|_{L^2}\nonumber\\
\leq&C\|u\|_{H^1}+C\|\nabla_h\omega\|_{L^2}+C\|\partial_z\omega\|_{L^2}
+C\|\Lambda_h^{1+\epsilon}\Lambda_v^{1-\epsilon}u\|^{\beta}_{L^2}\|\Lambda_h^{1+\epsilon}\Lambda_v^{2-\epsilon}u\|^{1-\beta}_{L^2}\nonumber\\
\leq&C\|u\|_{H^1}+C\|\nabla_h\omega\|_{L^2}+C\|\partial_z\omega_\theta\|_{L^2}
+C\|\nabla_h\omega\|^{\beta}_{L^2}\|\nabla_h\partial_z\omega\|^{1-\beta}_{L^2}\nonumber\\
\leq&C\|u\|_{H^1}+C\|\nabla_h\omega\|_{L^2}+C\|\partial_z\omega_\theta\|_{L^2}
+C\|\nabla_h\omega\|^{\beta}_{L^2}\|\nabla_h\partial_z\omega_\theta\|^{1-\beta}_{L^2}\nonumber\\
&+C\|\nabla_h\omega\|^{\beta}_{L^2}\Big\| \partial_z\Big(\frac{\omega_\theta}{r}\Big)\Big\|^{1-\beta}_{L^2}.
\end{align*}
Collecting these estimates together with \eqref{eq.7.6} yields that for $\epsilon\in]0,\frac{1}{16}[,$
\begin{align}\label{eq.8.1}
\norm{\nabla_h u}_{L^1_tB_{\infty,\infty}^{\epsilon}}\leq&C\norm{u_{0}}_{H^1}+C\norm{u}_{L^{1}_{t}H^1}+C\|u\|^2_{L^2_tH^1}+C\|\partial_z\omega_\theta\|^{2}_{L^2_tL^2}\nonumber\\
&+\|\nabla_h\omega\|^{2}_{L^2_tL^2}+\norm{\rho}_{L^{1}_{t}L^2}
+\|\partial_z\omega_\theta\|^{2}_{L^2_tL^2}+\|\nabla\rho\|_{L^1_tL^2}\nonumber\\
&+C\|\nabla_h\omega\|^{2\beta}_{L^2_tL^2}\|\nabla_h\partial_z\omega_\theta\|^{2(1-\beta)}_{L^2_tL^2}+C\|\nabla_h\omega\|^{2\beta}_{L^2}\Big\| \partial_z\Big(\frac{\omega_\theta}{r}\Big)\Big\|^{2(1-\beta)}_{L^2}.
\end{align}
Plugging \eqref{eq.7.4} and \eqref{eq.7.6} into \eqref{eq.m-n}, and using  \mbox{Proposition \ref{energy}},  \mbox{Proposition \ref{prop-cruc}} and \mbox{Proposition \ref{log}}, we finally obtain that
\begin{align*}
&\|\partial_z\omega_\theta(t)\|^{2}_{L^2}+\|\nabla\rho(t)\|_{L^2}^2+\int_0^t\|\nabla_h\partial_z\omega_\theta(\tau)\|^{2}_{L^2}\intd\tau
+\int_0^t\Big\|\frac{\partial_z\omega_\theta(\tau)}{r^{2}}\Big\|^{2}_{L^2}\intd\tau\nonumber\\
\leq&  Ce^{e^{\exp{Ct^{17}}}}+Ce^{e^{\exp{Ct^{17}}}}\Big(Ct\|u\|_{L^\infty_tH^1}+C\sqrt{t}\|\partial_z\omega\|_{L^2_tL^2}+C\sqrt{t}C\|\nabla_h \partial_z\omega_\theta\|_{L^2_tL^2}\nonumber\\
&+C\sqrt{t}\Big\| \partial_z\Big(\frac{\omega_\theta}{r}\Big)\Big\|_{L^2_tL^2}+C\norm{u_{0}}_{H^1}+C\norm{u}_{L^{1}_{t}H^1}+C\|u\|^2_{L^2_tH^1}+C\|\partial_z\omega_\theta\|^{2}_{L^2_tL^2}\\
&+\|\nabla_h\omega\|^{2}_{L^2_tL^2}+\norm{\rho}_{L^{1}_{t}L^2}
+\|\partial_z\omega_\theta\|^{2}_{L^2_tL^2}+\|\nabla\rho\|_{L^1_tL^2}\\
&+C\|\nabla_h\omega\|^{2\beta}_{L^2_tL^2}\|\nabla_h\partial_z\omega_\theta\|^{2(1-\beta)}_{L^2_tL^2}+C\|\nabla_h\omega\|^{2\beta}_{L^2}\Big\| \partial_z\Big(\frac{\omega_\theta}{r}\Big)\Big\|^{2(1-\beta)}_{L^2}\Big)\\
\leq&
Ce^{e^{\exp{Ct^{17}}}}+Ce^{e^{\exp{Ct^{17}}}}\int_0^t\|\partial_z\omega_\theta(\tau)\|^{2}_{L^2}\intd\tau+C\int_0^t\|\nabla\rho(\tau)\|_{L^2}^2\intd\tau+
\frac14\int_0^t\|\nabla_h\partial_z\omega_\theta(\tau)\|^{2}_{L^2}\intd\tau
\\&+\frac14\int_0^t\Big\|\frac{\partial_z\omega_\theta(\tau)}{r^{2}}\Big\|^{2}_{L^2}\intd\tau.
\end{align*}
This implies
\begin{align*}
\|\partial_z\omega_\theta(t)\|^{2}_{L^2}+\|\nabla\rho(t)\|_{L^2}^2\leq Ce^{e^{\exp{Ct^{17}}}}+Ce^{e^{\exp{Ct^{17}}}}\int_0^t\|\partial_z\omega_\theta(\tau)\|^{2}_{L^2}\intd\tau+C\int_0^t\|\nabla\rho(\tau)\|^2_{L^2}\intd\tau.
\end{align*}
Thus the Gronwall inequality enables us to get the desired result \eqref{high-es}. In addition, from \eqref{eq.7.4} and \eqref{eq.8.1}, we can obtain the second desired result \eqref{high-es-2} in view of \mbox{Proposition \ref{energy}}, \mbox{Proposition \ref{prop-cruc}} and \mbox{Proposition \ref{prop-pri}}.
\end{proof}
\section{Proof of Theorem \ref{global}}\label{Sec-4}
\setcounter{section}{4}\setcounter{equation}{0}
In this section, we restrict our attention to show Theorem \ref{global}. Before proving, we first give an useful proposition which palys an important role in proof of Theorem \ref{global}.
\begin{proposition}\label{res-h}
Let $(u_0,\rho_0)\in H^{s+1}\times H^{s}$ with $s>\frac32$. Assume that $\text{\rm div}u_0=0$. Then \eqref{bs} admits a unique global solution $(\rho,u)$ satisfying $u\in C(\RR^+;H^{s+1})$ and $\rho\in C(\RR^+;H^{s})$.
\end{proposition}

Here we adopt the classical Friedrichs method (see for example \cite{bcd}) to prove the existence part of \mbox{Proposition \ref{res-h}}. For $n\geq 1$, the spectral cut-off $J_n$ be defined by
\begin{equation*}
\widehat{J_{n}f}(\xi)=\chi_{[0,n]}(|\xi|)\widehat{f}(\xi), \quad\text{for each}\quad \xi \in\RR^3.
\end{equation*}
Thus, by the same argument as in \cite{MZ2012}, one can consider the following system in the spaces $L^{2}_{n}:=\{f\in L^2(\RR^{3})|\,\text{ Supp} f\subset B(0,n)\}$:
\begin{equation*}\label{approx}
  \left\{\begin{array}{ll}
\partial_tu_{n}+\mathcal{P}J_{n}\text{div}(u_{n}\otimes u_{n})-\Delta_{h}u_{n}=\mathcal{P}J_n(\rho_{n} e_3),\quad (t,x)\in\RR^+\times\RR^3,\\
\partial_t\rho_{n} +J_n\text{div}(u_{n} \rho_{n})=0,\\
\text{div}\,u_{n}=0,\\
(\rho_{n},u_{n})|_{t=0}=J_{n}(\rho_0,u_0).
\end{array}\right.
\end{equation*}
Note that the operators $J_n$ and $\mathcal{P}J_n$ are the orthogonal projectors for the $L^2$-inner product. Combining this  with the stability result of \cite[Lemma 5.1]{A-H-K0} ensures that the above formal calculations remain unchanged.

Next, our first target is to show the local well-posedness for the system \eqref{bs}. Applying the operator $\Delta_{q}$ to the density equation, we thus get
\begin{equation*}
\partial_t\Delta_q\rho+S_{q+1}u\cdot \nabla\Delta_q \rho=S_{q+1}u\cdot \nabla\Delta_q \rho-\Delta_{q}(u\cdot\nabla\rho):=F_{q}(u,\rho).
\end{equation*}

Taking the $L^{2}$-inner product to the above equality with $\Delta_q\rho$ and using the incompressible condition, we thus obtain
\begin{equation}\label{h-density}
\frac12\frac{\mathrm{d}}{\mathrm{d}t}\norm{\Delta_{q}\rho(t)}^{2}_{L^{2}}\leq\norm{F_{q}(u,\rho)}_{L^{2}}\norm{\Delta_{q}\rho}_{L^{2}}.
\end{equation}
By using Lemma \ref{commutator-est}, multiplying both sides by $2^{2qs}$ and summing up over $q\geq-1$, we have
\begin{equation}\label{density-h}
\frac{\mathrm{d}}{\mathrm{d}t}\norm{\rho(t)}^{2}_{H^{s}}\leq C\Big(\norm{\nabla u(t)}_{L^{\infty}}\norm{\rho(t)}^{2}_{H^{s}}+\norm{\rho(t)}_{\infty}\norm{\omega(t)}_{H^{s}}\norm{\rho(t)}_{H^{s}}\Big).
\end{equation}
Next, we turn to show the estimate of $\omega$. Applying $\Delta_{q}$ to the velocity equations, we get
\begin{equation*}
\partial_t\Delta_{q}u+S_{q+1}u\cdot\nabla \Delta_{q}u-\Delta_{h}\Delta_{q}u+\Delta_q\nabla\Pi=S_{q+1}u\cdot \nabla\Delta_q u-\Delta_{q}(u\cdot\nabla u):=F_{q}(u,u).
\end{equation*}
Taking the $L^{2}$-inner product to this equation with $\Delta_{q}u$ and using the divergence free condition, we can conclude that
\begin{equation*}
\begin{split}
&\frac12\frac{\mathrm{d}}{\mathrm{d}t}\norm{\Delta_{q}u(t)}^{2}_{L^{2}}+\norm{\nabla_{h}\Delta_{q}u(t)}^{2}_{L^{2}}
\leq\norm{F_{q}(u,u)}_{L^{2}}\norm{\Delta_{q}u}_{L^{2}}.
\end{split}
\end{equation*}
By using Lemma \ref{commutator-est} again, multiplying both sides by $2^{2q(1+s)}$ and summing up over $q\geq-1$, we have
\begin{equation}\label{density-h}
\frac{\mathrm{d}}{\mathrm{d}t}\norm{u(t)}^{2}_{H^{1+s}}\leq C\norm{\nabla u(t)}_{L^{\infty}}\norm{u(t)}^{2}_{H^{1+s}}.
\end{equation}
In a similar way as above, we can conclude that the approximate solution $(\rho_n, u_n)$ to \eqref{approx} satisfies
\begin{equation*}
\norm{(\rho_{n},\omega_{n})(t))}^{2}_{H^{s}}+\int^{t}_{0}\norm{\nabla_h\omega_{n}(\tau)}^{2}_{H^{s}}\mathrm{d}\tau\leq \norm{J_{n}(\rho_{0},\omega_{0}))}^{2}_{H^s}e^{t}e^{C\int_{0}^{t}\norm{(\rho_{n},\nabla u_{n})(\tau)}_{\infty}\mathrm{d}\tau}.
\end{equation*}
Since $s>\frac32$, the space $H^s(\mathbb{R}^3)$ continuously embeds in $L^{\infty}(\mathbb{R}^3)$. Thus, the well-known fact that $\norm{\nabla u_n}_{H^s}$ is equivalent to $\norm{\omega_n}_{H^s}$ entails  us to conclude that
$$X_{n}(t)\leq \norm{(\rho_0,\omega_0)}^{2}_{H^s}e^{\frac{t}{2}}e^{C\int_{0}^{t}X_{n}(\tau)\mathrm{d}\tau}\text{ and }X^{2}_{n}(t):=\norm{(\rho_n,\omega_n)(t)}^{2}_{H^s}.
$$
This inequality may be easily integrated into
\begin{equation*}
\exp \Big(-C\int_0^tX_n(\tau)\mathrm{d}\tau\Big)\geq 1-2CX_{0}e^{\frac{t}{2}},\quad \text{for all } t\geq 0.
\end{equation*}
The energy estimate yields the $L^2$-bound of $u_n$. Therefore, there exists a constant $c>0$ such that if we set
\begin{equation}\label{lifespan}
T:=2\log\Big(\frac{c}{\norm{(\rho_0,\omega_0)}_{H^s}}\Big).
\end{equation}
Therefore, we get
\begin{equation}\label{uniform}
\rho_n\in L^{\infty}([0,T[;H^s),\quad u_n\in L^{\infty}([0,T[;H^{s+1})\quad\text{and}\quad \nabla_hu_n\in L^{2}([0,T[;H^{s+1}).
\end{equation}
We now turn to proof of the local existence of a solution. By virtue of \mbox{Equations \eqref{approx}} and uniform estimetes \eqref{uniform}, it is easy to check that
$\partial_{t}\rho_n\in L^{\infty}([0,T[;H^{s-1})$ and $\partial_{t}u_n\in L^{2}([0,T[;H^{s})$. on the other hand, we know that $H^{s}\hookrightarrow H^{s-1}$
and $H^{s+1}\hookrightarrow H^{s}$ are locally compact. Therefore, by the classical Aubin-Lions argument and Cantor's diagonal process, we can conclude that
 there exists a solution $(\rho,u)$ in $L^{\infty}([0,T[;H^s\times H^{s+1})$ such that $\nabla_h u\in L^{2}([0,T[;H^{s+1})$.
 The time continuity follows from the fact that $\rho$ and $\omega$ satisfy transport equations with $H^{s}$ initial data and a $L^{2}([0,T[;H^{s})$ source
 term. In addition, the standard energy method allows us to obtain the uniqueness of solution for Lipschitz vector field.

Now, it remains for us to show that the local smooth solutions may be extended to all positive time.
Put together the lower bound for the lifespan of $(\rho,u)$ give by \eqref{lifespan} and the uniqueness of smooth solutions, it suffices to state that under the assumption of the theorem, we have
\begin{equation}\label{local-1}
\sup_{0\leq t\leq T}(\norm{\rho(t)}_{H^{s}}+\norm{\omega(t)}_{H^{s}})<\infty.
\end{equation}
First, as $\rho$ is transported by the vector-fields $u$ (which is Lipschitz for $s>\frac32$ implies $H^{s}\hookrightarrow L^{\infty}$),
we get
\begin{equation*}
\norm{\rho(t)}_{L^{\infty}}\leq\norm{\rho_0}_{L^{\infty}}\quad \text{for all } t\in [0,T[.
\end{equation*}
In consequence, the energy estimate \eqref{high-es} ensures that
\begin{equation}\label{local-2}
\norm{(\rho,\omega)(t))}_{H^{s}}+\int^{t}_{0}\norm{\nabla_h\omega(\tau)}^{2}_{H^{s}}\mathrm{d}\tau\leq \norm{(\rho_{0},\omega_{0}))}^{2}_{H^s}e^{t}e^{C\int_{0}^{t}\norm{\nabla u(\tau)}_{\infty}\mathrm{d}\tau}.
\end{equation}
On the other hand, by virtue of \mbox{Proposition \ref{energy}}, \mbox{Proposition \ref{prop-cruc}} and \mbox{Lemma \ref{log}}, we can deduce that
\begin{align}\label{local-3}
\int^{t}_{0}\norm{\nabla u(\tau)}_{L^\infty}\mathrm{d}\tau\leq &C+\sup_{2\leq q<\infty}\int^{t}_{0}\frac{\norm{S_{q}\nabla u(\tau)}_{L^{\infty}}}{\sqrt{q}}\mathrm{d}\tau\Big(\log\big(e+\norm{\omega}_{L^{1}_{T}(H^{s})}\big)\Big)^{\frac12}\nonumber\\
\leq&C+\Big(\sup_{2\leq p<\infty}\int^{t}_{0}\frac{\norm{\nabla u(\tau)}_{L^{p}}}{\sqrt{p}}\mathrm{d}\tau\Big)^{2}+\log\big(e+\norm{\omega}_{L^{1}_{T}(H^{s})}\big)\Big.
\end{align}
Inserting \eqref{local-3} in \eqref{local-2} gives
\begin{align*}
&\norm{(\rho,\omega)(t)}_{H^{s}}+\int^{t}_{0}\norm{\nabla_h\omega(\tau)}^{2}_{H^{s}}\mathrm{d}\tau\nonumber \\
\leq&\norm{(\rho_{0},\omega_{0}))}^{2}_{H^s}e^{Ct}e^{C\big(\sup_{2\leq p<\infty}\int^{t}_{0}\frac{\norm{\nabla u(\tau)}_{L^{p}}}{\sqrt{p}}\mathrm{d}\tau\big)^2}
\Big(e+\int_0^t\norm{\omega(\tau)}_{H^{s}}\mathrm{d}\tau\Big),
\end{align*}
which together with the Gronwall inequality and \mbox{Proposition \ref{prop-fine}} gives the desired  result \eqref{local-1}. This completes the proof of \mbox{Proposition \ref{res-h}}.

Now, let us turn to prove \mbox{Theorem \ref{global}}. We first construct the following approximate scheme:
\begin{equation}
      \label{full-bs-appp}
\left\{ \begin{array}{ll}
(\partial_{t}+u\cdot\nabla) u^n-\Delta_h u^n+\nabla \Pi^n=\rho^n e_3,\quad (t,x)\in \RR^+\times\RR^3,\\
(\partial_{t}+u\cdot\nabla)\rho^n=0,\\
\textnormal{div}\,  u^n=0,\\
  (u^n,\rho^n)|_{t=0}=(S_{n+1}u_{0},S_{n+1}\rho_{0}).
\end{array} \right.
\end{equation}
Since $u_0^n, \,\rho_0^n\in H^\infty:=\bigcap_{s>0} H^s$, we know that \eqref{full-bs-appp} has a unique global solution $(\rho^n,u^n)$ by taking advantage of \mbox{Proposition \ref{res-h}}. Thus, \mbox{Proposition \ref{energy}}, \mbox{Proposition \ref{prop-cruc}} and \mbox{Proposition \ref{prop-pri}} ensure
\begin{align*}
&\rho^n\in L^\infty_{\rm loc}(\RR^+;H^1\cap L^\infty),\quad u^n\in L^\infty_{\rm loc}(\RR^+;H^1),\quad \nabla_hu^n\in L^2_{\rm loc}(\RR^+;H^1),\\
&\partial_z\omega^n\in  L^\infty_{\rm loc}(\RR^+;L^2),\quad \nabla_h\partial_z\omega^n\in  L^2_{\rm loc}(\RR^+;L^2)\quad \text{and}\quad\nabla u\in L^1_{\rm loc}(\RR^+;L^\infty).
\end{align*}
Mimicking the compactness argument used for proving \mbox{Proposition \ref{res-h}}, one can
conclude that there exists a solution $(\rho,u)$ such that (deduced from the Fatou lemma)
\begin{align*}
&\rho\in L^\infty_{\rm loc}(\RR^+;H^1\cap L^\infty),\quad u\in L^\infty_{\rm loc}(\RR^+;H^1),\quad \nabla_hu\in L^2_{\rm loc}(\RR^+;H^1),\\
&\partial_z\omega\in  L^\infty_{\rm loc}(\RR^+;L^2),\quad \nabla_h\partial_z\omega\in  L^2_{\rm loc}(\RR^+;L^2)\quad \text{and}\quad\nabla u\in L^1_{\rm loc}(\RR^+;L^\infty).
\end{align*}
From the above estimates, we eventually obtain the time continuity by the same argument as used in \cite[Proposition F.4]{MZ2012}.

It remains for us to show the uniqueness statement. Let $(\rho,u,\Pi)$ and $(\widetilde\rho,\widetilde u,\widetilde\Pi)$ be  two solutions of the
 system \eqref{bs} with the same initial \mbox{data $(u_0,\rho_0)$} such that
 \begin{align}\label{eq-ddi}
&\rho,\tilde\rho\in L^\infty_{\rm loc}(\RR^+;H^1\cap L^\infty),\quad u,\,\tilde u\in L^\infty_{\rm loc}(\RR^+;H^1),\quad \nabla_hu,\,\nabla_h\tilde u\in L^2_{\rm loc}(\RR^+;H^1),\nonumber\\
&\partial_z\omega,\,\partial_z\tilde\omega\in  L^\infty_{\rm loc}(\RR^+;L^2),\quad \nabla_h\partial_z\omega,\,\nabla_h\partial_z\tilde\omega\in  L^2_{\rm loc}(\RR^+;L^2)\quad \text{and}\quad\nabla u,\,\nabla\tilde u\in L^1_{\rm loc}(\RR^+;L^\infty).
\end{align}
   Then the difference $(\delta\rho,\delta u,\delta \Pi)$ between two solutions $(\rho,u,\Pi)$ and $(\widetilde\rho,\widetilde u,\widetilde\Pi)$ satisfies
\begin{equation*}
\begin{cases}
\partial_t\delta u+u\cdot\nabla\delta u-\Delta_{h}\delta u+\nabla\delta \Pi=\delta\rho e_{3}+\delta u\cdot\nabla \widetilde u,\\
\partial_t\delta\rho+u\cdot\nabla\delta\rho=-\delta u\cdot\nabla\widetilde \rho.
\end{cases}
\end{equation*}
The standard energy argument together with the fact from the Plancherel theorem that
\begin{align*}
\int_{\RR^3}\delta\rho e_3\delta u\intd x=\int_{\RR^3}\delta\rho \delta u_z\intd x=&\int_{\RR^3}(\Delta)^{-1}\delta\rho \Delta\delta u_z\intd x\\
=&\int_{\RR^3}(\Delta)^{-1}\delta\rho \Delta_h\delta u_z\intd x-\sum_{i=1}^2\int_{\RR^3}(\Delta)^{-1}\delta\rho \partial_z \partial_i\delta u_i\intd x\\
\leq&C\|\delta\rho\|_{H^{-1}}\|\nabla_h\delta u\|_{L^2}
\end{align*}
 enables us to infer that
\begin{align*}
\frac12\dtd \|\delta u(t)\|_{L^2}^2+\|\nabla_h\delta u(t)\|_{L^2}^2\leq C\|\delta\rho\|_{H^{-1}}\|\nabla_h\delta u\|_{L^2}+\|\nabla\tilde u\|_{L^\infty}\|\delta u\|^{2}_{L^2}
\end{align*}
which implies
\begin{align}\label{eq-d1}
\dtd \|\delta u(t)\|_{L^2}\leq C\|\delta\rho\|_{H^{-1}}+\|\nabla\tilde u\|_{L^\infty}\|\delta u\|_{L^2}.
\end{align}
And the Fourier localization technique and the classical commutator estimate (see for example \cite{CWZ2012}) allow us to conclude that
\begin{equation}\label{eq-d2}
\dtd \|\delta\rho(t)\|_{H^{-1}}\leq C\|\nabla u\|_{L^\infty}\|\delta \rho\|_{H^{-1}}+\|\delta u\|_{L^2}\|\tilde \rho\|_{L^\infty}.
\end{equation}
Putting \eqref{eq-d1} together with  \eqref{eq-ddi} and \eqref{eq-d2}, and using the Gronwall inequality entails $(\delta\rho,\delta u)\equiv0$. This completes the proof of \mbox{Theorem \ref{global}}.

\appendix
\section{Appendix}
\label{appendix}
\setcounter{section}{5}\setcounter{equation}{0}
In this section, we shall give two useful lemmas which have been used throughout the paper.
\begin{lemma}[\cite{bcd}]\label{lema.2}
There exists the positive constant $C$ such that
\begin{equation}\label{a.5}
\begin{split}
\int_{\mathbb{R}^{3}}fgh\mathrm{d}x_{1}\mathrm{d}x_{2}\mathrm{d}x_{3}
\leq C\norm{f}^{\frac{1}{2}}_{L^{2}}\norm{\partial_{x_{3}}f}^{\frac{1}{2}}_{L^{2}}
\norm{g}^{\frac{1}{2}}_{L^{2}}\norm{\nabla_{h}g}^{\frac{1}{2}}_{L^{2}}\norm{h}^{\frac{1}{2}}_{L^{2}}\norm{\nabla_{h}h}^{\frac{1}{2}}_{L^{2}}.
\end{split}
\end{equation}
\end{lemma}
\begin{lemma}\label{commutator-est}
Let $s>-1$ and $1\leq p\leq\infty$. Assume that $u$ be a divergence free vector fields over $\mathbb{R}^d$. There exists a positive constant $C$ such that for all $q\geq-1$
\begin{equation}\label{eq-l}
2^{qs}\big\|R_q(u,u)\big\|_{L^2}\leq c_q\|\nabla u\|_{L^\infty}\|u\|_{B^s_{p,r}}\quad\text{with}\quad c_q\in l^2,
\end{equation}
and
\begin{equation}\label{eq-last}
\big\|R_q(u,\rho)\big\|_{L^2}\leq C\Big(\|\nabla u\|_{L^\infty}\sum_{q'\geq q-4}2^{q-q'}\|\Delta_{q'}\rho\|_{L^2}+\|\rho\|_{L^\infty}\sum_{|q-q'|\leq4}\|\Delta_{q'}\nabla  u\|_{L^2}\Big),
\end{equation}
where $R_q(u,v):= S_{q+1}u\cdot\nabla\Delta_qv-\Delta_q(u\cdot\nabla v)$.
\end{lemma}
\begin{proof}
We just give the proof of \eqref{eq-last},  since the proof of \eqref{eq-l} is standard.
First of all, one  decomposes $R_q(u,\rho)$ as follows:
\begin{align*}
R_q(u,\rho)=&u\cdot\nabla\Delta_q\rho-\Delta_q(u\cdot\nabla \rho)-\Delta_q\big(({\rm I_d}-S_{q+1})u\cdot\nabla \rho\big)\nonumber\\
=&-[\Delta_q,S_{q+1}\bar u]\cdot\nabla \rho-[\Delta_q,S_1u]\cdot\nabla \rho-\Delta_q\big(({\rm I_d}-S_{q+1})u\cdot\nabla\rho\big),
\end{align*}
where $\bar u=({\rm I_d}-S_1)u$.

Note that
\begin{align*}
[\Delta_q,S_{q+1}\bar u]\cdot\nabla\rho=&[\Delta_q,S_{q+1}T_{\bar u_i}]\partial_i\rho+\Delta_q\big(T_{\partial_i\rho}S_{q+1}\bar u_i\big)+
\Delta_q\big(R(S_{q+1}\bar u_i,\partial_i \rho)\big)\nonumber\\
&-T_{\Delta_q\partial_i\rho}S_{q+1}\bar u_i-R(S_{q+1}\bar u_i,\Delta_q\partial_i\rho)\nonumber\\
:=&R^{1}_q(u,\rho)+R^{2}_q(u,\rho)+R^{3}_q(u,\rho)+R^{4}_q(u,\rho)+R^{5}_q(u,\rho),
\end{align*}
and
\begin{align*}
\Delta_q\big(({\rm I_d}-S_{q+1})u\cdot\nabla \rho\big)=&\Delta_q\big(T_{({\rm I_d}-S_{q+1})u_i}\partial_i\rho\big)+\Delta_q\big(T_{\partial_i\rho}({\rm I_d}-S_{q+1})u_i\big)+
\Delta_qR\big(({\rm I_d}-S_{q+1})u_i,\partial_i\rho\big)\nonumber\\
:=&R^{6}_q(u,\rho)+R^{7}_q(u,\rho)+R^{8}_q(u,\rho).
\end{align*}
From above, it is clear to find that the only term $[\Delta_q,S_1u]\cdot\nabla \rho$ involves low frequencies of $u$.

First of all, we observe that
\begin{align*}
&[S_{q'-1}S_{q+1}\bar u_i,\Delta_q]\partial_i\Delta_{q'}\rho\nonumber\\
=&2^{qd}\int_{\mathbb{R}^d}\big(S_{q'-1}S_{q+1}\bar u_i(x)-S_{q'-1}S_{q+1}\bar u_i(x-y)\big)\varphi\big(2^{q}(x-y)\big)\partial_i\Delta_{q'}\rho(y)\intd y\nonumber\\
=&-2^{qd}\int_{\mathbb{R}^d}\int_0^1\partial_kS_{q'-1}S_{q+1}\bar u_i\big(x+(1-\tau)(x-y)\big)\intd\tau(x_k-y_k)\varphi\big(2^{q}(x-y)\big)\partial_i\Delta_{q'}\rho(y)\intd y\nonumber\\
=&-2^{q(d-1)}\int_{\mathbb{R}^d}\int_0^1\partial_kS_{q'-1}S_{q+1}\bar u_i\big(x+(1-\tau)(x-y)\big)\intd\tau2^{q}(x_k-y_k)\varphi\big(2^{q}(x-y)\big)\partial_i\Delta_{q'}\rho(y)\intd y,
\end{align*}
where used the relation $\Delta_{q}f=2^{qd}\int_{\mathbb{R}^d}\varphi\big(2^{q}(x-y)\big)f(y)\intd y$.

Therefore, we immediately get that
\begin{align}
\|R^{1}_q(u,\rho)\|_{L^p}\leq&C \sum_{|q'-q|\leq4}2^{-(q-q')}\|\partial_k S_{q'-1}\bar u_i\|_{L^\infty}
\|\partial_i\Delta_{q'}\rho\|_{L^p}\int_{\mathbb{R}^d}|x\varphi(x)|\intd x\nonumber\\
\leq&C \sum_{|q'-q|\leq4}2^{-(q-q')}\|\partial_k S_{q'-1}u_i\|_{L^\infty}\|\Delta_{q'}\rho\|_{L^p}\nonumber\\
\leq&C \|\nabla u\|_{L^\infty}\sum_{|q'-q|\leq4}2^{-(q-q')}\|\Delta_{q'}\rho\|_{L^p}.
\end{align}
In a similar fashion as for proving $R^1_q(u,\rho)$, we can bounded $[\Delta_q,S_1u]\cdot\nabla \rho$ as follows:
\begin{align}
\|[\Delta_q,S_1u]\cdot\nabla \rho\|_{L^p}\leq& C \sum_{|q'-q|\leq4}2^{-(q-q')}\|\partial_k S_{q'-1}u_i\|_{L^\infty}\|\Delta_{q'}\rho\|_{L^p}\nonumber\\
\leq&C \|\nabla u\|_{L^\infty}\sum_{|q'-q|\leq4}2^{-(q-q')}\|\Delta_{q'}\rho\|_{L^p}.
\end{align}
For the second term $R_q^2(u,\rho)$, the H\"older inequality allows us to conclude that
\begin{align}
\|R_q^2(u,\rho)\|_{L^p}\leq&C\sum_{|q'-q|\leq4}\|\Delta_{q'}\bar u_i\|_{L^p}\big\|S_{q'-1}\partial_i\rho\big\|_{L^\infty}\nonumber\\
\leq&C\big\|\rho\big\|_{L^\infty}\sum_{|q'-q|\leq4}\|\Delta_{q'}(\nabla u)\|_{L^p}
\end{align}
Similarly, we can conclude that
\begin{equation}
\|R_q^4(u,\rho)\|_{L^p}\leq C \|\nabla u\|_{L^\infty}\sum_{-1\leq q'\leq q+2}2^{q'-q}\big\|\Delta_{q'}\rho\big\|_{L^p}.
\end{equation}
The reminder term $R_q^3(u,\rho)$ can be bounded by
\begin{align}
\big\|\partial_i\Delta_q\big(R(S_{q+1}\bar u_i, \rho)\big)\big\|_{L^p}
\leq& C\sum_{q'\geq q-2}2^{q}\|\Delta_{q'}S_{q-1}\rho\|_{L^p}\|\tilde\Delta_{q'} \bar u_i\|_{L^\infty}\nonumber\\
\leq& C\sum_{q'\geq q-2}2^{q-q'}\|\Delta_{q'}S_{q-1}\rho\|_{L^p}\|\tilde\Delta_{q'} \nabla u\|_{L^\infty}\nonumber\\
\leq&C\|\nabla u\|_{L^\infty}\sum_{q'\geq q-2}2^{q-q'}\|\Delta_{q'}\rho\|_{L^p}.
\end{align}
In a similar way, one has
\begin{equation}
\|R_q^5(u,\rho)\|_{L^p}
\leq C\|\nabla u\|_{L^\infty}\sum_{q'\geq q-2}2^{q-q'}\|\Delta_{q'}\rho\|_{L^p}.
\end{equation}
It remain for us to bound the last three terms $R_q^6(u,\rho),\,R_q^7(u,\rho)$ and $R_q^8(u,\rho)$. Thanks to the property of support and the H\"older inequality, one has
\begin{align}
\|R_q^6(u,\rho)\|_{L^p}
\leq &C\sum_{|q'-q|\leq4}\|S_{q'-1}({\rm I_d}-S_{q+1})u_i\|_{L^\infty}\|\Delta_{q'}\partial_i\rho\|_{L^p}\nonumber\\
\leq&C\sum_{|q'-q|\leq4}2^{-q}\|S_{q'-1}({\rm I_d}-S_{q+1})\nabla u_i\|_{L^\infty}\|\Delta_{q'}\partial_i\rho\|_{L^p}\nonumber\\
\leq&C\sum_{|q'-q|\leq4}2^{q'-q}\|S_{q'-1}\nabla u_i\|_{L^\infty}\|\Delta_{q'}\rho\|_{L^p}\nonumber\\
\leq&C \|\nabla u\|_{L^\infty}\sum_{|q'-q|\leq4}2^{q'-q}\|\Delta_{q'}\rho\|_{L^p}.
\end{align}
For the term $R_q^7(u,\rho)$, by the H\"older inequality, we obtain
\begin{align}
\|R_q^7(u,\rho)\|_{L^p}
\leq &C\sum_{|q'-q|\leq4}\|S_{q'-1}\partial_i\rho\|_{L^\infty}\|\Delta_{q'}({\rm I_d}-S_{q+1})u_i\|_{L^p}\nonumber\\
\leq&C\sum_{|q'-q|\leq4}2^{q'-q}\|S_{q'-1}\rho\|_{L^\infty}\|\Delta_{q'}({\rm I_d}-S_{q+1})\nabla u_i\|_{L^p}\nonumber\\
\leq&C\|\rho\|_{L^\infty}\sum_{|q'-q|\leq4}2^{q'-q}\|\Delta_{q'}({\rm I_d}-S_{q+1})\nabla u\|_{L^p}.
\end{align}
As for the last term $R_q^8(u,\rho)$, by the H\"older inequality, we obtain
\begin{align}
\|R_q^8(u,\rho)\|_{L^p}
\leq &C\big\|\partial_i\Delta_qR\big(({\rm I_d}-S_{q+1})u_i,\rho\big)\big\|_{L^p}\nonumber\\
\leq&C\sum_{q'\geq q-2}2^{q}\|\Delta_{q'}\rho\|_{L^p}\|\tilde\Delta_{q'}({\rm I_d}-S_{q+1})u_i\|_{L^\infty}\nonumber\\
\leq&C\sum_{q'\geq q-2}2^{q}\|\Delta_{q'}\rho\|_{L^p}\|\tilde{\dot{\Delta}}_{q'} u_i\|_{L^\infty}\nonumber\\
\leq&C\sum_{q'\geq q-2}2^{q-q'}\|\Delta_{q'}\rho\|_{L^p}\|\tilde{\dot{\Delta}}_{q'}\nabla u_i\|_{L^\infty}\nonumber\\
\leq&C\|\nabla u\|_{L^\infty}\sum_{q'\geq q-2}2^{q-q'}\|\Delta_{q'}\rho\|_{L^p}.
\end{align}
Combining these estimates yields the desired result \eqref{eq-last}.
\end{proof}

\vskip .2in
\section*{Acknowledgements}
The authors are thankful to Prof. M. Paicu for informing us the references \cite{DP2008,DP2008-2}.
The authors were supported by NSF of China under grants No.11171033 and No.11231006.

 \end{document}